\documentclass{siamart1116}
\usepackage{amsmath, amssymb, bm, mathrsfs}
\usepackage{color, graphicx,subcaption,accents,cite}
  \makeatletter
\def\widebar{\accentset{{\cc@style\underline{\mskip10mu}}}}
\makeatother

\def\theenumi{\alph{enumi})}
\def\labelenumi{\theenumi}
\def\theenumii{(\roman{enumii})}
\def\labelenumii{\theenumii}

 \makeatletter
\renewcommand{\theequation}{%
	\thesection.\arabic{equation}}
\@addtoreset{equation}{section}
\makeatother

\usepackage{lipsum}
\usepackage{amsfonts}
\usepackage{graphicx}
\usepackage{epstopdf}
\usepackage{algorithmic}
\ifpdf
  \DeclareGraphicsExtensions{.eps,.pdf,.png,.jpg}
\else
  \DeclareGraphicsExtensions{.eps}
\fi

\numberwithin{theorem}{section}

\newcommand{\TheTitle}{Strong Stability
of Sampled-data  Riesz-spectral Systems} 
\newcommand{\TheAuthors}{}

\headers{Strong Stability
	of Sampled-data Systems}{\TheAuthors}

\title{{\TheTitle}\thanks{Submitted to the editors DATE.
\funding{This work was supported by JSPS KAKENHI Grant Number JP20K14362.}}}

\author{
	Masashi Wakaiki\thanks{Graduate School of System Informatics, Kobe University, Nada, Kobe, Hyogo 657-8501, Japan
		(\email{wakaiki@ruby.kobe-u.ac.jp}).}
}

\usepackage{amsopn}

\ifpdf
\hypersetup{
  pdftitle={\TheTitle},
  pdfauthor={\TheAuthors}
}
\fi

\newtheorem{assumption}[theorem]{Assumption}
\newtheorem{example}[theorem]{Example}
\newtheorem{remark}[theorem]{Remark}

\newcommand{\im}{\mathop{\rm Im}\nolimits}

\newcommand{\re}{\mathop{\rm Re}\nolimits}

\newcommand{\ran}{\mathop{\rm ran}\nolimits}

\begin{document}

\maketitle

\begin{abstract}
Suppose that a continuous-time linear infinite-dimensional system
with a static state-feedback controller is strongly stable.
We address the following question:
If we convert the continuous-time controller to
a sampled-data controller by applying an idealized sampler and
a zero-order hold, will the resulting sampled-data system be
strongly stable for all sufficiently small sampling periods?
In this paper, we restrict our attention to the situation where
the generator of the open-loop system is a Riesz-spectral operator
and its point spectrum has a limit point at the origin.
We present conditions under which the answer to 
the above question is affirmative.
In the robustness analysis, we show that 
the sufficient condition for strong stability obtained in 
the Arendt-Batty-Lyubich-V\~u theorem
is preserved between 
the original continuous-time system and
the sampled-data system
under fast sampling.
\end{abstract}



\section{Introduction}
We consider systems with state space $X$ and input space $\mathbb{C}$
of the form
\begin{equation}
\label{eq:plant_intro}
	\dot x(t) = Ax(t) + Bu(t),\quad t\geq 0;\qquad x(0) = x^0 \in X,
\end{equation}
where $X$ is a Hilbert space, $A$ is the generator of a strongly
continuous semigroup $(T(t) )_{t\geq 0}$ on $X$, and
$B$ is a bounded linear operator from $\mathbb{C}$ to $X$.
Suppose that a continuous-time feedback control $u(t) = Fx(t)$,
where $F$ is a bounded linear operator from $X$ to $\mathbb{C}$,
achieves the strong stability of the closed-loop system in the sense that 
$A+BF$ generates a strongly stable semigroup
$(T_{BF}(t) )_{t\geq 0}$ on $X$, i.e., 
\[
\lim_{t\to \infty}\|T_{BF}(t) x^0\| = 0\qquad \forall x^0 \in X.
\]
Instead of this continuous-time controller, we use
the following digital controller with an idealized sampler and a zero-order hold:
\begin{equation}
\label{eq:controller_intro}
u(t) = Fx(k\tau),\quad k\tau \leq t < (k+1)\tau,
\end{equation}
where $\tau>0$ is the sampling period.
If the sampling period $\tau$ is sufficiently small, then
the control input $u$ generated by the digital controller can be
almost identical to the one generated by the continuous-time controller.
Therefore, we would expect that the sampled-data 
system \eqref{eq:plant_intro} and \eqref{eq:controller_intro}
is also strongly stable in the sense that 
\[
\lim_{t\to \infty} \|x(t)\| = 0 \qquad \forall x^0 \in X.
\]
Our objective is to show that a certain class of infinite-dimensional systems possess
this  robustness property with respect to sampling.

In the finite-dimensional case,
stability is preserved for all sufficiently small sampling periods.
This result has been extended to the exponential stability 
of some classes of infinite-dimensional systems in \cite{Logemann2003,Rebarber2006}, but
even exponential stability is much more delicate in the infinite-dimensional
case \cite{Rebarber2002}.
Sampled-data systems are ubiquitous in computer-based control systems, and
various sampled-data control problems have been
studied for infinite-dimensional systems; for example, stabilization 
\cite{Tarn1988, Rebarber1998, Logemann2005, Logemann2013, Karafyllis2018,Kang2018Automatica,Wakaiki2020SCL,
	Lin2020} and output regulation 
\cite{Logemann1997, Ke2009SCL, Ke2009SIAM, Ke2009IEEE, Wakaiki2019}.
Robustness of strong stability with respect to sampling has been
posed as an open problem in \cite{Rebarber2006MTNS}, and it has not been solved yet.

Strong stability of strongly continuous semigroups is rather weak compared with
exponential stability. In fact, exponential stability is preserved under all sufficiently
small bounded perturbations,
whereas it is easy to find a strongly stable semigroup and an arbitrarily 
small perturbation such that the perturbed semigroup is unstable; see, e.g.,
Section~1 of \cite{Paunonen2011}.
The difficulty of robustness analysis of strong stability arises from
the high level of generality of strong stability. 
For example, define 
the bounded linear operator $A_1$ on $\ell^2(\mathbb{C})$
by
\begin{equation}
\label{eq:A1_def}
A_1 x := \sum_{n=1}^{\infty} 
-\frac{1}{n}
\langle x,\phi_n \rangle \phi_n
\end{equation}
and the operator $A_2$ on  $\ell^2(\mathbb{C})$  by
\begin{equation}
\label{eq:A2_def}
A_2 x := \sum_{n=1}^{\infty} \left(
-\frac{1}{n} + in
\right) \langle x,\phi_n \rangle \phi_n
\end{equation}
with domain
\[
D(A_2) := \left\{
x \in \ell^2(\mathbb{\mathbb{C}}) :
\sum_{n=1}^{\infty} n^2 |\langle x , \phi_n \rangle |^2 < \infty
\right\},
\]
where $\{\phi_n:n \in \mathbb{N}\}$ 
is the standard basis of $\ell^2(\mathbb{\mathbb{C}})$. Both operators $A_1$ and $A_2$
generate strongly stable semigroups. However,
the behaviors of the semigroups and the spectral properties of $A_1$ and $A_2$
are quite  different.
Focusing on important subclasses of strongly stable semigroups,
the author of \cite{Paunonen2011,
	Paunonen2012SS,Paunonen2013SS,Paunonen2014JDE,Paunonen2015,Paunonen2015Springer} 
has studied robustness of strong
stability. To study strong stability of delay semigroups, 
perturbation results for strongly stable semigroups have been developed in \cite{Rastogi2020}.
The preservation of strong stability under discretization via the Cayley transformation has been
investigated in \cite{Guo2006, Besseling2010}.
We can regard discretization by sampling as a perturbation, but
this structured perturbation has not been investigated in the above previous studies.

In this paper, we concentrate on the situation where 
the system \eqref{eq:plant_intro} is a Riesz-spectral system, i.e., 
the generator $A$
is a Riesz-spectral operator; see Definition~\ref{def:RSO} below for
the definition of Riesz-spectral operators.
We further assume that $A$ has no eigenvalues on 
the imaginary axis and only
finitely many eigenvalues in
$\{\lambda \in \mathbb{C}\setminus \{ 0\}: \re \lambda > -\alpha,~|\arg \lambda| < \pi/2 + \delta \}$ (the gray area in Fig.~\ref{fig:C_Sigma})
for some $\alpha >0$ and $0<\delta \leq \pi/2$ but that 
there exists a sequence of the eigenvalues of $A$ such that 
it is contained in the sector
$\{\lambda \in \mathbb{C}\setminus \{ 0\}: \pi/2 + \delta \leq |\arg \lambda| \leq \pi \}$ 
and converges to $0$. Consequently, $0$ belongs to  the continuous spectrum of $A$.
For example, the operator $A_1$ given in \eqref{eq:A1_def} 
has this spectral property, but $A_2$ in \eqref{eq:A2_def} does not.
The sectorial constraint on the eigenvalues
avoids any losses of high-frequency information 
caused by sampling. 
Such a sectorial constraint on the spectrum of the generator $A$
has been also placed in the previous study \cite{Logemann2003} in order to 
prove that  exponential stability is preserved under fast sampling
in the case of boundary or pointwise control.
\begin{figure}[tb]
	\centering
	\includegraphics[width = 5cm]{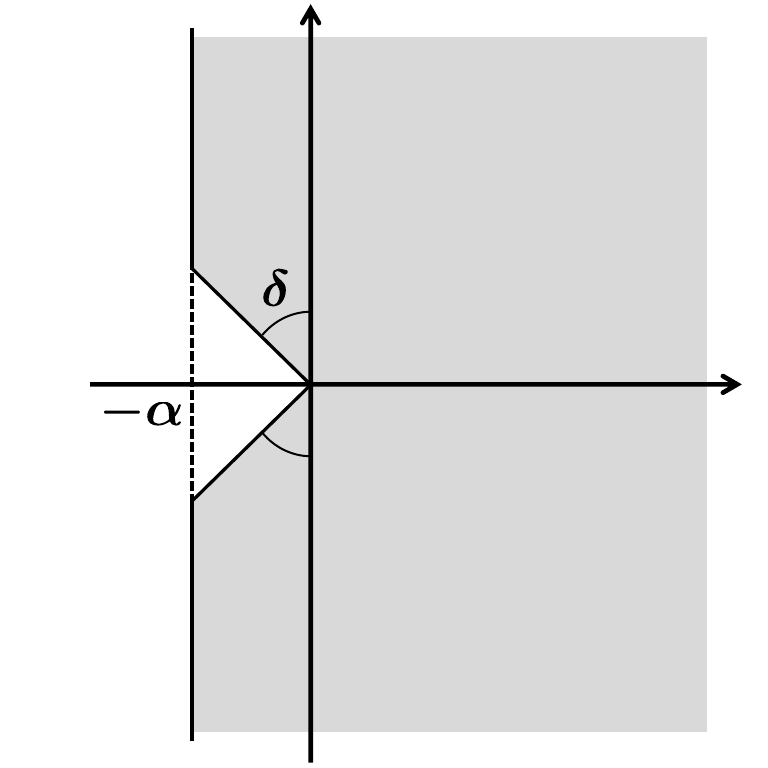}
	\caption{Set $\{\lambda \in \mathbb{C}
		\setminus \{ 0\}: \re \lambda > -\alpha,~|\arg \lambda| < \pi/2 + \delta \}$.}
	\label{fig:C_Sigma}
\end{figure}

Another important assumption of this study is that $A+BF$ 
satisfies the sufficient condition for strong stability obtained 
in the well-known  Arendt-Batty-Lyubich-V\~u theorem~\cite{Arendt1998,Lyubich1988}, that is,
$\sup_{t\geq 0} \|T_{BF}(t)\| < \infty$, $\sigma_p(A+BF) \cap i \mathbb{R} = \emptyset$, 
and $\sigma(A+BF) \cap i \mathbb{R} = \{0\}$, where $\sigma_p(A+BF)$ and $\sigma(A+BF)$
denote the point spectrum and the spectrum of $A+BF$, respectively.
It is straightforward to show that 
the sampled-data system \eqref{eq:plant_intro} and \eqref{eq:controller_intro} is 
strongly stable if and only if the discrete semigroup $(\Delta(\tau)^k)_{k \in \mathbb{N}}$ on
$X$, where 
\[
\Delta(\tau) := T(\tau) + \int^\tau_0 T(s)BF ds,
\]
is strongly stable, i.e., 
\[
\lim_{k \to \infty} \|\Delta(\tau)^kx^0 \| = 0\qquad \forall x^0 \in X;
\]
see Section~\ref{sec:inf_dim_SD_sys}.
Then the robustness analysis of strong stability with respect to sampling
becomes the problem of determining whether or not 
the discrete semigroup $(\Delta(\tau)^k)_{k \in \mathbb{N}}$ 
is strongly stable for all sufficiently small $\tau>0$.
To check the strong stability of $(\Delta(\tau)^k)_{k \in \mathbb{N}}$, we use the discrete version of 
the Arendt-Batty-Lyubich-V\~u theorem.
More precisely, we prove that 
$\sup_{k \in \mathbb{N}} \|\Delta(\tau)^k\| < \infty$, 
$\sigma_p(\Delta(\tau)) \cap \mathbb{T} = \emptyset$, and
$\sigma(\Delta(\tau)) \cap \mathbb{T} = \{1\}$.
In summary, we here show that 
if the continuous-time closed-loop operator $A+BF$
satisfies the sufficient condition for strong stability 
in the continuous case 
of the Arendt-Batty-Lyubich-V\~u theorem, then
the discretized closed-loop operator $\Delta(\tau)$
also satisfies the sufficient condition of the discrete counterpart
for all sufficiently small sampling periods $\tau >0$.
This means that 
the sufficient condition for strong stability obtained in 
 the Arendt-Batty-Lyubich-V\~u theorem
 is preserved between the original continuous-time system
 and the sampled-data system under fast sampling.

This paper is organized as follows.
In Section~\ref{sec:inf_dim_SD_sys}, we first review
useful results on strong stability and Riesz-spectral operators and then
state our main result on robustness of strong stability with respect to sampling.
To prove this result, we study the spectrum of $\Delta(\tau)$ in Section~\ref{sec:spectrum}.
In Section~\ref{sec:power_boundedness}, we investigate
the boundedness of the discrete semigroup $(\Delta(\tau)^k)_{k \in \mathbb{N}}$
in order to complete the proof of the main result.
Concluding remarks are made in Section~\ref{sec:conclusion}.

\subsection*{Notation and terminology}
For $\alpha \in \mathbb{R}$ and $r>0$, we define
\[
\mathbb{C}_\alpha := 
\{
\lambda  \in \mathbb{C}: \re \lambda  > \alpha
\},\quad 
\mathbb{D}_r := 
\{
\lambda  \in \mathbb{C}: |\lambda | <r
\},\quad 
\mathbb{E}_r := 
\{
\lambda  \in \mathbb{C}: |\lambda | > r
\}.
\]
We denote the unit circle by $\mathbb{T} := \{
\lambda  \in \mathbb{C} : |\lambda | = 1
\}$ and
the imaginary axis by $i\mathbb{R} := 
\{ 
i \omega : \omega \in \mathbb{R}
\}$.
For $\delta \in (0,\pi]$, we define the sector
$\Sigma_{\delta} := 
\{
\lambda  \in \mathbb{C} \setminus \{0\} :
|\arg \lambda | < \delta
 \}$.
Let $X$ and $Y$ be Banach spaces. For a linear operator $A:X\to Y$,
we denote by $D(A)$, $\ran (A)$, and $\ker (A)$ the domain, the range, and
the kernel of $A$, respectively. 
The space of all bounded linear operators from $X$ to $Y$ is denoted by $\mathcal{L}(X,Y)$,
and we define $\mathcal{L}(X) := \mathcal{L}(X,X)$.
For a linear operator $A: D(A) \subset X \to X$, we denote by 
$\sigma(A)$, $\sigma_p(A)$, and $\rho(A)$ the spectrum, the point spectrum, and
the resolvent set of $A$, respectively.
The resolvent operator is denoted by $R(\lambda,A) = (\lambda I - A)^{-1}$
for $\lambda \in \rho (A)$.
For a set $S \subset X$ and a linear operator 
$A: D(A) \subset X \to Y$, we write for $A|_S$ 
the restriction of $A$ to $S$, i.e., $A|_S x = Ax$ with domain
$D(A|_S) := D(A) \cap S$.	
If $X$ is a Hilbert space, then we denote 
the inner product by $\langle x,y\rangle$ for $x,y \in X$ and
the Hilbert space adjoint
by $A^*$ for a linear operator $A$ with dense domain in $X$.
Sequences $\{\phi_n:n \in \mathbb{N}\}$ and  
$\{\psi_n:n \in \mathbb{N}\}$ on a Hilbert space 
are called {\em biorthogonal} if
\[
\langle \phi_n, \psi_m \rangle = 
\begin{cases}
1 & \text{if $n=m$} \\
0 & \text{otherwise}.
\end{cases}
\] 

Let $X$, $U$, and $Y$ be Banach spaces,
$A$ generate a strongly continuous semigroup on $X$, 
$B \in \mathcal{L}(U,X)$, $C \in \mathcal{L}(X,Y)$, and $\beta \in \mathbb{R}$.
The control system $(A,B,-)$ is called {\em $\beta$-exponentially stabilizable} if
there exists $F \in \mathcal{L}(X,U)$ such that the growth bound of 
the semigroup generated by $A+BF$ is less than $\beta$.
If $(A,B,-)$ is $0$-exponentially stabilizable, then it is called 
{\em exponentially stabilizable}.
The control system $(A,-,C)$ is called 
{\em$\beta$-exponentially detectable} if 
there exists $L \in \mathcal{L}(Y,X)$ such that the growth bound of 
the semigroup generated by $A+LC$ is less than $\beta$.
If $(A,-,C)$ is $0$-exponentially detectable, then it is called 
{\em exponentially detectable}.
A strongly continuous  semigroup $(T(t))_{t\geq 0}$ 
on $X$ is called {\em uniformly bounded}  if $\sup_{t \geq 0} \|T(t)\| <\infty$
and {\em strongly stable} if 
$\lim_{t\to \infty} T(t)x = 0
$ for every $x \in X$.
By a {\em discrete semigroup} on $X$, we mean a family $(\Delta^k)_{k \in \mathbb{N}}$
of operators, where $\Delta \in \mathcal{L}(X)$.
A discrete semigroup 
$(\Delta^k )_{k \in \mathbb{N}}$ on $X$ is called 
	{\em power bounded}  
if $\sup_{k \in \mathbb{N}} \|\Delta^k\| <\infty$ and
	{\em strongly stable}  
	if $\lim_{k\to \infty} \|\Delta^kx\| = 0$ for every $x \in X$.

\section{Infinite-dimensional sample-data system}
\label{sec:inf_dim_SD_sys}
Let $X$ be a Hilbert space, and
consider the following sampled-data system with 
state space $X$: 
\begin{subequations}
	\label{eq:sampled_data_sys}
	\begin{align}
	\dot x(t) &= Ax(t) + Bu(t),\quad t\geq 0;\qquad x(0) = x^0 \in X \\
	u(t) &= Fx(k\tau),\quad k\tau \leq t < (k+1)\tau,
	\end{align}
\end{subequations}
where 
$x(t) \in X$ is the state, $u(t) \in \mathbb{C}$ is the control input, 
$\tau>0$ is the sampling period,
$A: D(A) \subset X \to X$ generates a strongly continuous semigroup 
$(T(t) )_{t\geq 0}$ on $X$,
$B \in \mathcal{L}(\mathbb{C},X)$ is the control operator, and
$F \in \mathcal{L}(X,\mathbb{C})$ is the feedback operator.

\begin{definition}
	{\em
	The sampled-data system \eqref{eq:sampled_data_sys} is called
	{\em strongly stable} if 
	\[\lim_{t\to \infty} \|x(t)\| = 0
	\] 
	for every initial  state $x^0 \in X$.
}
\end{definition}

The objective of this paper is to show that
if the strongly continuous semigroup $(T_{BF}(t) )_{t\geq 0}$ 
generated by $A+BF$
is strongly stable,
then the sampled-data system \eqref{eq:sampled_data_sys}
is also strongly stable for all sufficiently small sampling periods $\tau>0$.

For $t\geq 0$,
define $S(t) \in \mathcal{L}(\mathbb{C},X)$ and $\Delta(t) \in \mathcal{L}(X)$ by
\begin{equation}
\label{eq:S_def}
S(t) := \int^t_0 T(s) Bds,\quad 
\Delta(t) := T(t) + S(t) F,
\end{equation}
respectively.
Then the state $x$ of 
the sampled-data system \eqref{eq:sampled_data_sys}
satisfies
\begin{equation}
\label{eq:discretized_sys}
x\big((k+1)\tau\big) = \Delta(\tau) x(k\tau)\qquad \forall k\in \mathbb{N} \cup \{0\}.
\end{equation}

By
the following proposition,
it suffices to investigate the strong stability of the discrete semigroup
$(\Delta(\tau)^k)_{k \in \mathbb{N}}$
in order to study the strong stability of 
the sampled-data system \eqref{eq:sampled_data_sys}.
\begin{proposition}
	\label{prop:SD_DT}
	The sampled-data system \eqref{eq:sampled_data_sys} is strongly stable
	if and only if the discrete semigroup $(\Delta(\tau)^k)_{k \in \mathbb{N}}$
	is strongly stable.
\end{proposition}
\begin{proof}
	Since $(\Rightarrow)$ immediately follows from 
	\eqref{eq:discretized_sys}, 
	we here show only $(\Leftarrow)$.
	Suppose that $(\Delta(\tau)^k)_{k \in \mathbb{N}}$
	is strongly stable. Let $x^0 \in X$ be given. We obtain
	\[
	x(k\tau + t) = \Delta(t) x(k\tau) = 
	\Delta(t) \Delta(\tau)^k x^0 \qquad 
	\forall t\in [0,\tau),~\forall k \in \mathbb{N}\cup \{0\}.
	\]
	By the strong continuity of $(T(t) )_{t\geq 0}$,
	there exists $c\geq 1$ such that 
	\[
	\|\Delta(t)\| \leq c\qquad \forall t\in [0,\tau).
	\]
	It follows that 
	\[
	\|x(k\tau + t)\| \leq c \|\Delta(\tau)^k x^0\| \qquad 
	\forall t\in [0,\tau),~\forall k \in \mathbb{N}\cup \{0\}.
	\]
	By assumption, $\|\Delta(\tau)^k x^0\| \to 0$ as $k \to \infty$. Thus,
	we obtain $x(t) \to 0$ as $t\to \infty$.
\end{proof}

Instead of dealing with strong stability directly,
we employ the following sufficient conditions
obtained in the Arendt-Batty-Lyubich-V\~u theorem~\cite{Arendt1998,Lyubich1988}.

\begin{theorem}[Continuous case]
	\label{thm:ABLV_cont}
	Let $(T(t))_{t\geq 0}$ be a uniformly bounded 
	semigroup generated by $A$ on a Hilbert space.
	If $\sigma_p(A) \cap i\mathbb{R} = \emptyset$ and if $\sigma(A) \cap i\mathbb{R}$ is countable,
	then $(T(t))_{t\geq 0}$ is strongly stable.
\end{theorem}
\begin{theorem}[Discrete case]
	\label{thm:ABLV_disc}
	Let $(\Delta^k)_{k \in \mathbb{N}}$ be a power bounded 
	discrete semigroup on a Hilbert space.
	If $\sigma_p(\Delta) \cap \mathbb{T} = \emptyset$ and if $\sigma(\Delta) \cap \mathbb{T}$ is countable,
	then $(\Delta^k)_{k \in \mathbb{N}}$ is strongly stable.
\end{theorem}

\subsection{Basic facts on Riesz-spectral operators}
In the sampled-data system \eqref{eq:sampled_data_sys}, we assume that 
$A$ is a Riesz-spectral operator,
which is defined as follows.
\begin{definition}[Definition 3.2.6 of \cite{Curtain2020}]
	\label{def:RSO}
	{\em
	Let $A$ be a closed linear operator on a Hilbert space $X$ with
	simple eigenvalues $\{\lambda_n:n \in \mathbb{N}\}$ and
	corresponding eigenvectors $\{\phi_n:n \in \mathbb{N}\}$. 
	We say that $A$ is a {\em Riesz-spectral operator} if
	the following two conditions are satisfied:
	\begin{enumerate}
		\item $\{\phi_n:n \in \mathbb{N}\}$ is a {\em Riesz basis} for $X$, that is,
		\begin{enumerate}
			\item 
			the closed linear span of $\{\phi_n:n \in \mathbb{N}\} $ is $X$; and
			\item there exist constants $M_{\rm a}, M_{\rm b} >0$ such that
			for all $N \in \mathbb{N}$ and all $a_n \in \mathbb{C}$, 
			$1 \leq n \leq N$,
			\begin{equation}
			\label{eq:RB_constant}
			M_{\rm a} \sum_{n=1}^N |a_n|^2 \leq 
			\left\|
			\sum_{n=1}^N a_n \phi_n
			\right\|^2 \leq 
			M_{\rm b}\sum_{n=1}^N |a_n|^2;
			\end{equation}
		\end{enumerate}
		\item 
		the set of eigenvalues
		$\{\lambda_n:n \in \mathbb{N}\}$ has
		at most finitely many accumulation points.
	\end{enumerate} 
}
\end{definition}

Before stating  the main result, we recall
some basic facts on Riesz bases and Riesz-spectral operators;
see \cite{Curtain2020, Tucsnak2009, Guo2019book} for more details. 
\begin{lemma}[Lemma~3.2.4 of \cite{Curtain2020}]
	\label{lem:RB_prop}
	Let  $\{\phi_n:n \in \mathbb{N}\}$ form a Riesz basis on a Hilbert 
	space $X$. Then the following hold: {\em
	\begin{enumerate}
		\item {\em 
			There exists a unique biorthogonal sequence 
			$\{\psi_n:n \in \mathbb{N}\}$, and
			$\{\psi_n:n \in \mathbb{N}\}$ is also a Riesz basis for $X$.
		}
		\item 
		{\em
		Every $x \in X$ can be represented uniquely by
		\begin{equation}
		\label{eq:x_representation}
		x = \sum_{n=1}^\infty \langle x, \psi_n \rangle \phi_n.
		\end{equation}
		Moreover, using constants $M_{\rm a}, M_{\rm b} >0$ satisfying
		\eqref{eq:RB_constant}, one has
		\[
		M_{\rm a} \sum_{n=1}^\infty |\langle x, \psi_n \rangle |^2 \leq \|x\|^2 \leq 
		M_{\rm b} \sum_{n=1}^\infty |\langle x, \psi_n \rangle |^2\qquad \forall x \in X.
		\]
	}
	\end{enumerate}
}
\end{lemma}

\begin{lemma}[Lemma~3.2.5 of \cite{Curtain2020}]
	Suppose that a closed linear 
	operator $A$ on a Hilbert space $X$
	has simple eigenvalues $\{\lambda_n:n \in \mathbb{N} \}$
	and that their corresponding eigenvectors $\{\phi_n:n \in \mathbb{N}\}$
	form a Riesz basis in $X$. If
	$\{\psi_n:n \in \mathbb{N}\}$ are the eigenvectors 
	of the adjoint $A^*$ of $A$ corresponding to the eigenvalues $\{ \overline{\lambda_n }:n \in \mathbb{N} \}$,
			then $\{\psi_n:n \in \mathbb{N}\}$ can be suitably scaled so that 
			$\{\phi_n:n \in \mathbb{N}\}$ and  
			$\{\psi_n:n \in \mathbb{N}\}$  are biorthogonal.
\end{lemma}

\begin{theorem}[Theorem~3.2.8 of \cite{Curtain2020}]
	\label{thm:RS_C0}
	Suppose that $A$ is a Riesz-spectral operator on a Hilbert space $X$
	with simple
	eigenvalues $\{\lambda_n:n \in \mathbb{N}\}$ and corresponding
	eigenvectors $\{\phi_n :n \in \mathbb{N}\}$.
	Let $\{\psi_n:n \in \mathbb{N}\}$ be the eigenvectors of $A^*$ such that 
	$\{\phi_n :n \in \mathbb{N}\}$ and $\{\psi_n:n \in \mathbb{N}\}$ are biorthogonal. Then
	$A$ has the following properties: {\em 
	\begin{enumerate}
		\item {\em $A$ satisfies 
		$\rho(A) = \{\lambda \in \mathbb{C}: 
		\inf_{n \in \mathbb{N}}|\lambda - \lambda_n| >0 \}$, 
		$\sigma(A) = \overline{\{\lambda_n:n \in \mathbb{N} \}} $, and
		\[
		(\lambda I - A)^{-1}x = \sum_{n=1}^{\infty} 
		\frac{1}{\lambda - \lambda_n} \langle
		x, \psi_n
		\rangle \phi_n\qquad \forall x\in X,~\forall \lambda \in \rho(A).
		\]}
		\item {\em $A$ has the representation 
		\[
		Ax = \sum^{\infty}_{n=1} \lambda_n 
		\langle x, \psi_n \rangle \phi_n \qquad \forall x \in D(A),
		\]
		and $D(A)$ can be written as
		\[
		D(A) = \left\{
		x \in X: \sum_{n=1}^{\infty} |\lambda_n|^2 \cdot |\langle 
		x, \psi_n \rangle
		|^2 < \infty
		\right\}.
		\]
	}
		\item {\em $A$ is the generator of 
		a strongly continuous semigroup $(T(t))_{t\geq 0}$ 
		if and only if $\sup_{n \in \mathbb{N}} \re \lambda_n < \infty$.
		The semigroup $(T(t) )_{t \geq 0}$ satisfies
		\begin{equation}
		\label{eq:T_rep}
		T(t)x = \sum_{n=1}^\infty e^{t \lambda_n} \langle
		x, \psi_n
		\rangle \phi_n\qquad \forall x \in X,~\forall t \geq 0,
		\end{equation}
		and the growth bound of 
		$(T(t))_{t\geq 0}$  is given by
		$\sup_{n\in \mathbb{N}} \re \lambda_n$. }
	\end{enumerate}
}
\end{theorem}

\subsection{Main result}
We place the following assumption on the sampled-data system.
\begin{assumption}
	\label{assum:for_MR}
	{\em 
	Let $A$ be a Riesz-spectral operator on a Hilbert space $X$ with
	simple eigenvalues $\{\lambda_n:n \in \mathbb{N}\}$ and
	corresponding eigenvectors $\{\phi_n:n \in \mathbb{N}\}$. 
	Let 
	$\{\psi_n: n \in \mathbb{N}\}$ be the eigenvectors of $A^*$  that 
	are biorthogonal with $\{\phi_n:n \in \mathbb{N}\}$.
	Let
	the control operator $B \in \mathcal{L}(\mathbb{C},X)$ 
	and the feedback operator $F\in \mathcal{L}(X,\mathbb{C})$ be
	represented as
	\begin{align}
	\label{eq:B_F_rep}
	Bu = bu,\quad u \in \mathbb{C};\qquad 
	Fx = \langle x, f\rangle,\quad x\in X
	\end{align}
	for some $b,f \in X$. Assume that 
	the operators $A$, $B$, and $F$ satisfy the following conditions:
	\begin{enumerate}
		\def\theenumi{\arabic{enumi}}
		\def\labelenumi{\theenumi}
		\renewcommand{\labelenumi}{(A\arabic{enumi})}
		\item
		\label{assump:finite_unstable}
		there exist $\alpha >0$ and $0 < \delta \leq \pi/2$ such that the set
		$\mathbb{C}_{-\alpha} \cap \Sigma_{\pi/2+\delta}$  has
		only finite elements of $\{\lambda_n:n \in \mathbb{N} \}$ (see also Fig.~\ref{fig:C_Sigma} for the set 
		$\mathbb{C}_{-\alpha} \cap \Sigma_{\pi/2+\delta}$);
		\item 
		\label{assump:imaginary} 
		$\{\lambda_n :n \in \mathbb{N}\} \cap i \mathbb{R} = \emptyset$;
		\vspace{3pt}
		\item
		\label{assump:origin}
		$0 \in \overline{\{\lambda_n:n \in \mathbb{N} \}}$;
		\item
		\label{assump:closed_loop}
		$A+BF$ generates a uniformly bounded semigroup on $X$
		and satisfies
		$\sigma_p(A+BF) \cap i \mathbb{R} = \emptyset$, 
		$\sigma(A+BF) \cap i \mathbb{R} = \{0\}$, and
		\begin{equation}
		\label{eq:imaginary_axis_RS_case}
		\sup_{\substack{\omega \in \mathbb{R}\\ |\omega| > 1}}
		\| R(i\omega, A+BF) \| < \infty;
		\end{equation}
		\item
		\label{assump:b_in_invA}
		$b$ satisfies 
		\[
		\sum_{n=1}^{\infty} 
		\left|\frac{\langle b, \psi_n \rangle}{\lambda_n} \right|^2 < \infty;
		\]
		\item
		\label{assump:b_F_cond}
		$b$ and $f$ satisfy 
		\[
		\sum^{\infty}_{n=1}
		\frac{\langle b, \psi_n \rangle  \langle \phi_n, f  \rangle}{ \lambda_n}  \not= -1.
		\]
	\end{enumerate}
}
\end{assumption}

\def\theenumi{\alph{enumi})}
\def\labelenumi{\theenumi}
\def\theenumii{(\roman{enumii})}
\def\labelenumii{\theenumii}

By (A\ref{assump:finite_unstable}),
$\sup_{n \in \mathbb{N}} \re \lambda_n < \infty$. Therefore, 
Theorem~\ref{thm:RS_C0}~c) shows that $A$ generates
a strongly continuous semigroup.
Since $\sigma(A) = \overline{\{\lambda_n:n \in \mathbb{N} \}} $ by
Theorem~\ref{thm:RS_C0}~a),
it follows from  (A\ref{assump:imaginary}) and (A\ref{assump:origin}) that 
$0 \in \sigma(A) \setminus \sigma_p(A)$.
Applying the mean ergodic theorem (see, e.g., Theorem~I.2.25 of \cite{Eisner2010}) to
the stable part of $A$, we find that $0$ belongs to the continuous spectrum of $A$;
see Remark \ref{rem:continuous_spectrum}  for details.
Note that the control system $(A,B,-)$ is not exponentially stabilizable by 
Theorem~8.2.3 of \cite{Curtain2020}.
By (A\ref{assump:closed_loop}) and
the Arendt-Batty-Lyubich-V\~u theorem,
the semigroup $( T_{BF}(t))_{t\geq 0}$ generated by $A+BF$ is
strongly stable.
Using the mean ergodic theorem again,
we see that  $0$ is still in the continuous spectrum of $A+BF$.

The assumption \eqref{eq:imaginary_axis_RS_case}
will be used to guarantee that 
$|1-FR(\lambda, A)B|$ is bounded from below by a positive constant
on $\overline{\mathbb{C}_0} \setminus \mathbb{D}_\eta$
for every $\eta>0$. 
This assumption \eqref{eq:imaginary_axis_RS_case}
appears also in the robustness analysis of strong stability developed 
in \cite{Paunonen2014JDE},
and
the assumption in the form
\[
\sup_{0 < |\omega| \leq 1} |\omega| \cdot
\| R(i\omega, A+BF) \| < \infty
\] 
is additionally placed in \cite{Paunonen2014JDE}. Instead of this assumption, we place
(A\ref{assump:b_in_invA}) and
(A\ref{assump:b_F_cond}) in order to obtain a positive lower bound of
$|1-FR(\lambda, A)B|$
on $(\overline{\mathbb{C}_0} \cap \mathbb{D}_\eta) \setminus \{0\}$
for a sufficiently small $\eta >0$.
The sectorial condition on the eigenvalues in (A\ref{assump:finite_unstable})
is also used for this purpose. 
Similarly, when the robustness of 
exponential stability with respect to sampling is analyzed 
for systems with unbounded control operators in \cite{Logemann2003},
the semigroup $(T(t))_{t \geq 0}$
is assumed to be holomorphic, which
guarantees that $\sigma(A)$ is contained in a certain sector.

We easily see that $b$ belongs to the domain of the 
algebraic inverse of $A$ under (A\ref{assump:b_in_invA}).
To deal with the high sensitivity of strong stability to perturbations,
we assume by (A\ref{assump:b_in_invA}) that the control operator $B$ has
the boundedness property related to the continuous spectrum of $A$
in addition to the standard boundedness property $B \in \mathcal{L}(\mathbb{C},X)$.
It is assumed also in \cite{Paunonen2014JDE}
 that 
perturbations have
such stronger boundedness properties.

The following theorem, which presents
robustness of strong stability with respect to sampling,
is the main result of this paper.
\begin{theorem}
	\label{thm:SD_SS}
	If Assumption~\ref{assum:for_MR} is satisfied, then
	there exists $\tau^*>0$ such that for every $\tau \in (0,\tau^*)$,
	the sampled-data system \eqref{eq:sampled_data_sys} is strongly stable. 
\end{theorem}

The idea of the proof 
is to show that  the discrete semigroup $(\Delta(\tau)^k)_{k \in \mathbb{N}}$
satisfies the sufficient condition for strong stability 
in the Arendt-Batty-Lyubich-V\~u theorem.
The sufficient condition given in this theorem consists of 
the spectral property and the boundedness property
of a discrete semigroup.
These properties of $(\Delta(\tau)^k)_{k \in \mathbb{N}}$ are investigated 
in Sections~\ref{sec:spectrum}  and \ref{sec:power_boundedness}, respectively.

\section{Spectrum and sampling}
\label{sec:spectrum}
Our first goal is 
to show that the  spectral properties in
the Arendt-Batty-Lyubich-V\~u theorem is satisfied 
for $\Delta(\tau)$ with sufficiently small $\tau>0$.
\begin{theorem}
	\label{thm:T_SF_resolvent_set}
	If Assumption~\ref{assum:for_MR} is satisfied, then
	there exists $\tau^*>0$ such that 
	$\sigma_p(\Delta(\tau)) \cap
	\mathbb{T}  = \emptyset$ and
	$
	\sigma(\Delta(\tau)) \cap
	\mathbb{T}  = \{1\} 
	$ for every $\tau \in (0,\tau^*)$.
\end{theorem}

With a slight modification of Theorem~\ref{thm:RS_C0}~a), we easily 
obtain the following properties of
the spectrum and the resolvent of $T(t)$ represented by \eqref{eq:T_rep}.
\begin{lemma}
	\label{lem:T_spectrum}
	Let $\{\lambda_n:n \in \mathbb{N}\} \subset \mathbb{C}$ satisfy 
	$\sup_{n \in \mathbb{N}}\re \lambda_n < \infty$.
	Suppose that $\{\phi_n :n \in \mathbb{N}\}$ is
	a Riesz basis for a Hilbert space $X$, and let
	$\{\psi_n:n \in \mathbb{N}\}$ be the 
	biorthogonal sequence to
	$\{\phi_n :n \in \mathbb{N}\}$. If we
	define $T(t) \in \mathcal{L}(X)$ by 
	\eqref{eq:T_rep} for $t\geq 0$, 
	then  
	$\sigma_p(T(t)) = \{e^{t \lambda_n} :n \in \mathbb{N}\}$ and 
	$\sigma(T(t)) = \overline{\{e^{t \lambda_n} :n \in \mathbb{N}\}}$
	for every $t\geq 0$.
	Moreover, for all $z \in \rho(T(t))$ and $t \geq 0$, the resolvent
	$R(z,T(t))$ is given by
	\begin{equation}
	\label{eq:resol_T}
	R\big(z,T(t)\big)x = \sum_{n=1}^{\infty} \frac{1}{z - e^{t \lambda_n}} \langle x, \psi_n \rangle \phi_n\qquad \forall x \in X.
	\end{equation}
\end{lemma}

We also immediately obtain
a representation of the algebraic inverse of a Riesz-spectral operator $A$
with $0 \not\in \sigma_p(A)$.
\begin{lemma}
	\label{lem:L_prop}
	Let $A$ be a Riesz-spectral operator on a Hilbert space $X$ 
	as in Theorem~\ref{thm:RS_C0}, and
	assume that the eigenvalues $\{\lambda_n:n \in \mathbb{N} \}$ satisfy
	$\sup_{n \in \mathbb{N}} \re \lambda_n < \infty$ and $\lambda_n \not= 0$
	for all $n \in \mathbb{N}$.
	The operator $A_0$ defined by 
	\begin{equation*}
	A_0x := \sum_{n=1}^\infty \frac{1}{\lambda_n} \langle 
	x, \psi_n
	\rangle \phi_n
	\end{equation*}
	with domain
	\begin{equation*}
	D(A_0) := \left\{
	x \in X :
	\sum_{n=1}^\infty 
	\left| \frac{
	\langle x, \psi_n \rangle
	 }{\lambda_n }\right|^2< \infty
	\right\}
	\end{equation*}
	is the algebraic inverse $A^{-1}$ of $A$, i.e., 
	satisfies $D(A_0) = \ran(A)$, 
	$A_0A x= x$ for every $x \in D(A)$, and
	$AA_0 x = x$ for every $x \in D(A_0)$. Moreover, 
	 for all
	$x \in D(A^{-1})$ and $t \geq 0$,
	the semigroup $(T(t) )_{t \geq 0}$ generated by 
	$A$ satisfies
	$T(t)x \in D(A^{-1})$ and $A^{-1}T(t) x = T(t) A^{-1} x$.
\end{lemma}

Lemma~\ref{lem:L_prop} shows that if
(A\ref{assump:b_in_invA}) also holds, i.e.,
 $b \in D(A^{-1})$, then
$S(t)$ defined by \eqref{eq:S_def} is written as
\begin{align}
\label{eq:S_rep}
S(t) 
=
A^{-1} (T(t) - I)B =  \sum_{n=1}^\infty \frac{e^{t \lambda_n} - 1}{\lambda_n} \langle b, \psi_n \rangle \phi_n \qquad \forall t \geq0.
\end{align}
This and  \eqref{eq:resol_T} yield
\begin{align}
\big(z I - T(t)\big)^{-1}S(t) 
&= \sum^{\infty}_{n=1}\frac{e^{t \lambda_n} - 1}{z - e^{t \lambda_n}} \cdot
\frac{\langle b , \psi_n \rangle}{\lambda_n} \phi_n\qquad \forall z \in \rho\big(T(t)\big),~
\forall t\geq 0.
\label{eq:Resol_S_tau}
\end{align}
Using the expression \eqref{eq:S_rep}, we obtain
the following simple result on the spectrum of $\Delta(t)$.
\begin{lemma}
	\label{lem:spectal_cond}
	Let $A$ be a Riesz-spectral operator on a Hilbert space $X$ whose eigenvalues
	$\{\lambda_n:n \in \mathbb{N} \}$ satisfy
	$\sup_{n \in \mathbb{N}} \re \lambda_n < \infty$, 
	{\rm (A\ref{assump:imaginary})}, 
	and {\rm (A\ref{assump:origin})}.
	Assume that $B \in \mathcal{L}(X,\mathbb{C})$ and
	$F \in \mathcal{L}(\mathbb{C},X)$ in the form of \eqref{eq:B_F_rep}
	satisfy
	{\rm (A\ref{assump:b_in_invA})} and {\rm (A\ref{assump:b_F_cond})}.
	Then $1 \in \sigma(\Delta(t)) \setminus 
	\sigma_p(\Delta(t))$
	for every $t>0$.
\end{lemma}
\begin{proof}
	Let $t>0$. 
	The essential spectrum of $T(t)$ contains $1$ under  (A\ref{assump:origin}), and
	it follows from the compactness of $S(t)F$  that 
	$T(t)$ and $\Delta(t) = T(t)+S(t)F$ have the same essential spectrum;
	see, e.g., 
	Section~IV.1.20 of \cite{Engel2000} and
	Section~XVII.2 of \cite{Gohberg1990}
	for an essential spectrum.
	Hence $1 \in \sigma(\Delta(t))$.

	To show that $1 \not\in \sigma_p(\Delta(t))$, 
	let $v \in X$ satisfy $\Delta(t)v =v$.
	By \eqref{eq:S_rep}, we obtain
	\[
	(T(t) - I) (I+A^{-1}BF)v = 0.
	\]
	Lemma~\ref{lem:T_spectrum} and (A\ref{assump:imaginary}) give
	$1 \not\in \sigma_p(T(t))$. Therefore,
	$(I+A^{-1}BF)v = 0$.
	If $Fv = 0$, then $v=0$.
	Otherwise, 
	\begin{equation}
	\label{FAinvB_cond}
	-FA^{-1}BFv= Fv \not=0.
	\end{equation}
	By (A\ref{assump:b_F_cond}), 
	\[
	FA^{-1}B = \langle A^{-1}b,f  \rangle = 
	\sum_{n=1}^\infty \frac{\langle b, \psi_n \rangle \langle  \phi_n, f \rangle}
	{\lambda_n} \not=-1,
	\]
	which contradicts \eqref{FAinvB_cond}.
	Thus, $v =0$; i.e.,
	$1 \not\in \sigma_p(\Delta(t))$.
\end{proof}

The next lemma provides a useful property of 
the spectrum of the product of bounded operators; see, e.g., 
(3) in Section~III.2 of \cite{Gohberg1990}.
\begin{lemma}
	\label{lem:bounded_operator_spectrum}
	Let $X$ and $Y$ be Banach spaces. For all
	$T \in \mathcal{L}(X,Y)$ and $S \in \mathcal{L}(Y,X)$, 
	$\sigma(TS) \setminus \{0\} = \sigma(ST) \setminus \{0\}$ holds.
\end{lemma}

We obtain the estimate of $|1 - F(\lambda I - A)^{-1}B| $ for 
$\lambda \in \rho(A) \cap \overline{\mathbb{C}_0}$.
\begin{lemma}
	\label{lem:continuous_time_trans_func_bound}
	Let $A$ be a Riesz-spectral operator on a Hilbert space $X$ whose eigenvalues
	$\{\lambda_n:n \in \mathbb{N} \}$ satisfy
	{\rm (A\ref{assump:finite_unstable})} and $0 
	\in \overline{\{\lambda_n:n \in \mathbb{N} \}} \setminus
	\{\lambda_n:n \in \mathbb{N} \}$.
	Assume  that $B \in \mathcal{L}(X,\mathbb{C})$ and
	$F \in \mathcal{L}(\mathbb{C},X)$ in the form of \eqref{eq:B_F_rep}
	satisfy
	{\rm (A\ref{assump:closed_loop})--(A\ref{assump:b_F_cond})}. Then
	there exists $\epsilon>0$ such that 
	$|1 - F(\lambda I - A)^{-1}B| > \epsilon$ for 
	every $\lambda \in \rho(A) \cap \overline{\mathbb{C}_0} $.
\end{lemma}

\begin{proof}
	We divide the proof into two cases.
	First, we consider the case where 
	$\lambda \in \rho(A)$ belongs to 
	$\overline{\mathbb{C}_0} \setminus \mathbb{D}_\eta$
	for $\eta >0$. 
	To this end, we employ the condition \eqref{eq:imaginary_axis_RS_case} in (A\ref{assump:closed_loop}).
	Second, we study the case $\lambda \in (\overline{\mathbb{C}_0} \cap \mathbb{D}_\eta) \setminus \{0\}$
	for a sufficiently small $\eta >0$, using (A\ref{assump:finite_unstable}),
	(A\ref{assump:b_in_invA}), and
	(A\ref{assump:b_F_cond}).

	Let $\lambda \in \rho (A)$ be given. Since
	\[
	\lambda I - A - BF = (\lambda I - A)(I - (\lambda I - A)^{-1}BF)
	\]
	and 
	since $\sigma((\lambda I - A)^{-1}BF) \setminus \{0 \} 
	= \sigma(F(\lambda I - A)^{-1}B) \setminus \{0 \} $
	by Lemma~\ref{lem:bounded_operator_spectrum},
	it follows that 
	\[
	\lambda \in  \rho (A+BF) \quad \Leftrightarrow \quad 
	1 \in \rho ((\lambda I - A)^{-1}BF) \quad \Leftrightarrow \quad 
	1 \in \rho (F(\lambda I - A)^{-1}B).
	\]
	Define $G(\lambda ) := F(\lambda I - A)^{-1}B$. A straightforward calculation shows that
	\begin{equation}
	\label{eq:extension}
	\frac{1}{1 - G(\lambda)} = F(\lambda I - A - BF)^{-1}B + 1\qquad \forall \lambda \in \rho(A) \cap \rho(A+BF).
	\end{equation}
	Using this equation, we can extend $1/(1 - G)$, which is defined only on $\rho(A) \cap \rho(A+BF)$,
	to a holomorphic function on $\rho(A+BF) \supset \overline{\mathbb{C}_0} \setminus \{0 \}$.
	
	\begin{figure}[tb]
	\centering
	\includegraphics[width = 5cm]{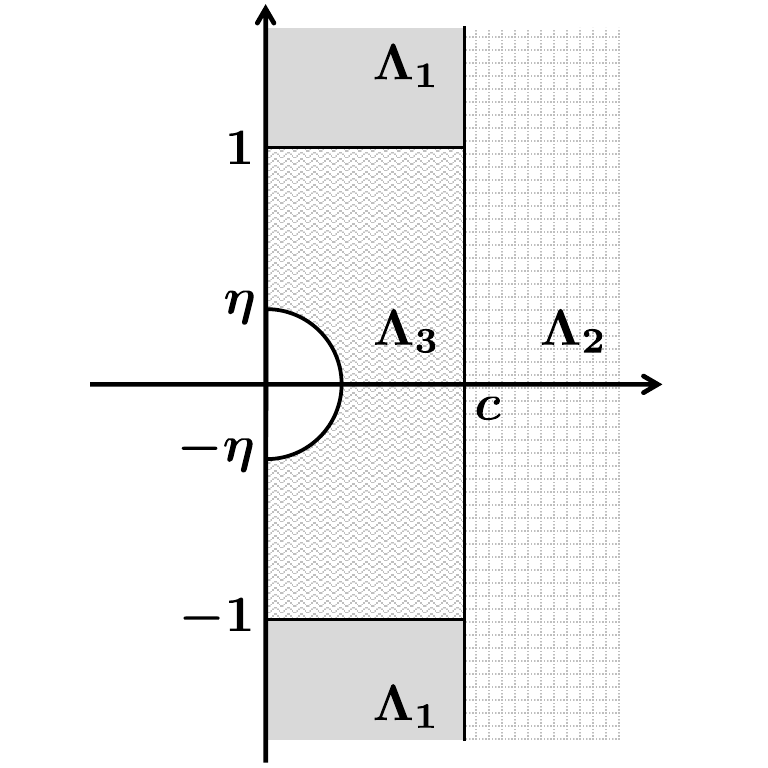}
	\caption{Sets $\Lambda_1$, $\Lambda_2$, and $\Lambda_3$.}
	\label{fig:Lambda}
\end{figure}
	Combining \eqref{eq:imaginary_axis_RS_case} in
	(A\ref{assump:closed_loop}) and 
	the Neumann series of resolvents (see, e.g., Proposition IV.1.3 of \cite{Engel2000}), 
	we have that $\|R(\lambda,A+BF)\| \leq 2M$
	for every $\lambda \in \Lambda_1 := \{\lambda \in \mathbb{C}: 0\leq \re \lambda \leq c, |\im \lambda | > 1\}$,
	where
	\[
	M := \sup_{\substack{\omega \in \mathbb{R}\\ |\omega| > 1}}
	\| R(i\omega, A+BF) \|,\qquad c := \frac{1}{2M}.
	\]
	By the uniform boundedness of $(T_{BF}(t))_{t\geq 0}$, the 
	Hille-Yosida theorem (see, e.g., Theorem II.3.8 of \cite{Engel2000}) shows that 
	\[
	\| R(\lambda, A+BF) \| \leq \frac{\sup_{t \geq 0}\|T_{BF}(t)\|}{c}
	\]
	for every $\lambda \in \Lambda_2 :=\mathbb{C}_c$.
	For $0 < \eta < \min\{c,1\}$,
	the resolvent $R(\lambda, A+BF)$ is holomorphic on the compact set 
	\[
	\Lambda_3 := 
	\{
	\lambda \in \mathbb{C}:
	|\lambda| \geq \eta,~0\leq \re \lambda \leq c,~|\im \lambda| \leq 1 
	\}.
	\]
	Therefore, 
	$\|R(\lambda,A+BF)\|$ is uniformly bounded on $\Lambda_3$.
	Since $\overline{\mathbb{C}_0} \setminus \mathbb{D}_\eta = 
\Lambda_1 \cup \Lambda_2 \cup \Lambda_3$ (see Fig.~\ref{fig:Lambda}),
we conclude from \eqref{eq:extension} that 
	\[
	\sup_{\lambda \in \overline{\mathbb{C}_0} \setminus \mathbb{D}_\eta}
	\left|\frac{1}{1 - G(\lambda)}\right| < \infty.
	\]

	By (A\ref{assump:finite_unstable}), there exists $\eta_1 >0$ 
	such that 
	\begin{equation}
	\label{eq:origin_neighborhood}
	\mathbb{D}_{\eta_1} \cap \Sigma_{\pi/2+\delta} \cap \{\lambda_n:
	n \in \mathbb{N}\}
	=\emptyset.
	\end{equation}
	Hence $(\overline{\mathbb{C}_0} \cap \mathbb{D}_{\eta_1}) 
	\setminus \{0\} \subset \rho(A)$.
	It remains to prove that there exist $\epsilon_1>0$ 
	and $\eta \in (0,\eta_1)$ such that
	\begin{equation}
	\label{eq:I_G_lower_bound}
	|1 - G(\lambda)| > \epsilon_1 \qquad \forall  \lambda \in (\overline{\mathbb{C}_0} \cap \mathbb{D}_{\eta}) 
	\setminus \{0\}.
	\end{equation}
	Note that
	$b \in D(A^{-1})$ by (A\ref{assump:b_in_invA}). It
	is enough to show that 
	$(re^{i\theta}- A)^{-1} b$ converges to $-A^{-1}b$ uniformly on $\theta \in [-\pi/2,\pi/2]$ as $r \to 0$.
	More precisely, we will show that for 
	every $\epsilon_2 >0$, there exists $\eta \in (0,\eta_1)$ such that 
	\[
	\|(re^{i\theta}- A)^{-1} b + A^{-1}b\|  < \epsilon_2
	\] 
	for all  $r \in (0,\eta)$ and all $\theta \in [-\pi/2, \pi/2]$. 
	Indeed, 
	using
	this fact and  $\epsilon_3  := |1 + FA^{-1}b|  >0$ 
	by  (A\ref{assump:b_F_cond}), we see that 
	\[
	|1 - G(re^{i\theta})| \geq |1 + FA^{-1}b| - 
	| F(re^{i\theta}- A)^{-1} b + FA^{-1}b| 
	> \epsilon_3 - \|F\|\epsilon_2 
	\]
	for all $r \in (0,\eta)$ and all $\theta \in [-\pi/2, \pi/2]$.
	Hence 
	if $\|F\|\epsilon_2 < \epsilon_3$, then
	\eqref{eq:I_G_lower_bound} holds 
	with $\epsilon_1 := \epsilon_3 - \|F\|\epsilon_2  >0$.

	Since 
	$r e^{i\theta} \in \rho(A)$
	for every $r \in (0,\eta_1)$ and 
	$\theta \in [-\pi/2, \pi/2]$, it follows from 
	Theorem~\ref{thm:RS_C0}~a) and Lemma~\ref{lem:L_prop}
	that
	\begin{align*}
	(re^{i\theta}- A)^{-1} b + A^{-1}b = 
	\sum_{n=1}^\infty
	\left( 
	\frac{1}{re^{i\theta} - \lambda_n} + \frac{1}{\lambda_n}\right)
	\langle b, \psi_n \rangle \phi_n
	\end{align*}
	for every $r \in (0,\eta_1)$ and 
	$\theta \in [-\pi/2, \pi/2]$.
	Using \eqref{eq:RB_constant}, we obtain
	\begin{align*}
	\|(re^{i\theta}- A)^{-1}b + A^{-1}b\|^2 &\leq 
	M_{\rm b}\sum_{n=1}^\infty
	\left|
	\left( 
	\frac{1}{re^{i\theta} - \lambda_n} + \frac{1}{\lambda_n}\right)
	\langle b, \psi_n \rangle
	\right|^2 \\
	&\leq
	M_{\rm b} \sum_{n=1}^\infty g_n(r)
	\qquad 
	\forall r \in (0,\eta_1),~
	\forall \theta \in \left[ -\frac{\pi}{2} , \frac{\pi}{2} \right],
	\end{align*}
	where
	\[
	g_n(r) := \sup_{-\pi/2 \leq \theta \leq \pi/2}\left|\left( 
	\frac{1}{re^{i\theta} - \lambda_n} + \frac{1}{\lambda_n}\right)
	\langle b, \psi_n \rangle
	\right|^2.
	\]
	
	Let $r \in (0, \eta_1/2)$ and $\theta \in [-\pi/2, \pi/2]$ be given.
	Suppose that 
	$\lambda _n \in \Sigma_{\pi/2+\delta}$. Then 
	$|\lambda_n| \geq \eta_1$ by \eqref{eq:origin_neighborhood}, and hence
	\begin{equation*}
	|re^{i\theta} - \lambda_n| \geq \frac{\eta_1}{2} > r.
	\end{equation*}
	Since 
	\begin{equation}
	\label{eq:lam_diff}
	\frac{1}{re^{i\theta} - \lambda_n} + \frac{1}{\lambda_n}= 
	\frac{re^{i\theta}}{re^{i\theta} - \lambda_n}\cdot \frac{1}{\lambda_n},
	\end{equation}
	it follows that
	\begin{equation}
	\label{eq:g_est1}
	g_n(r) \leq \left| \frac{\langle b, \psi_n \rangle}{\lambda_n} \right|^2
	\quad \text{if $\lambda_n \in \Sigma_{\pi/2+\delta}$}.
	\end{equation}
	Suppose next that  $\lambda _n \in \mathbb{C} \setminus 
	\Sigma_{\pi/2+\delta}$.
	Using the estimates
	\[
	 |\lambda_n| \leq \frac{|\re \lambda_n|}{\sin \delta},\quad
	|re^{i \theta} - \lambda_n| \geq |\re \lambda_n|,
	\]
	we obtain
	\begin{align*}
	\left|\frac{re^{i\theta}}{re^{i\theta} - \lambda_n} \right|
	\leq 
	1+ 
	\left|\frac{\lambda_n}{re^{i\theta} - \lambda_n} \right| \
	\leq 1 + \frac{1}{\sin \delta}. 
	\end{align*}
	By \eqref{eq:lam_diff}, 
	\begin{equation}
	\label{eq:g_est2}
	g_n(r) \leq 
	\left(
	1 + \frac{1}{\sin \delta}
	\right)^2
	\left| \frac{\langle b, \psi_n \rangle}{\lambda_n} \right|^2
	\quad \text{if $\lambda_n \in \mathbb{C} \setminus 
		\Sigma_{\pi/2+\delta}$}.
	\end{equation}
	
	The estimates \eqref{eq:g_est1} and \eqref{eq:g_est2} yield
	\[
	g_n(r) \leq 	\left(
	1 + \frac{1}{\sin \delta}
	\right)^2
	\left|\frac{\langle b ,\psi_n \rangle}{\lambda_n} \right|^2
	\qquad \forall r \in \left(
	0,\frac{\eta_1}{2}
	\right),~\forall n \in \mathbb{N}.
	\]
	Combining this with (A\ref{assump:b_in_invA}), we obtain
	\[
	\lim_{r \to 0} \sum_{n=1}^{\infty} g_n(r) =  \sum_{n=1}^{\infty} \lim_{r \to 0} g_n(r) =0.
	\]
	Thus, 
	$(re^{i\theta}- A)^{-1} b$ converges to $-A^{-1}b$ uniformly on $\theta \in [-\pi/2,\pi/2]$ as $r\to 0$. This completes the proof.
\end{proof}

\begin{remark}
	{\em
		Combining 
	\eqref{eq:extension} and Lemma~\ref{lem:continuous_time_trans_func_bound},
	we see that 
	the transfer function $G_{BF}(\lambda) := 
	F(\lambda I - A - BF)^{-1}B$ is holomorphic 
	and bounded on $\mathbb{C}_0$.
	This property is equivalent to
	the exponential stability of $(T_{BF}(t) )_{t\geq 0}$
	if $(A,B,F)$ is exponentially 
	stabilizable and exponentially 
	detectable; see, e.g., Theorem~VI. 8.35 of \cite{Engel2000}.
	However, since $0$ is an accumulation point of $\sigma_p(A)$  in our problem setting,
	$(A,B,F)$ is not exponentially 
	stabilizable or exponentially 
	detectable by Theorem~8.2.3 of \cite{Curtain2020}.
	}
\end{remark}

As in the robustness analysis of exponential stability with respect to sampling
given in Theorem~2.1 of
\cite{Rebarber2006},
we connect the estimates of the continuous-time system and
the discrete-time system.
\begin{lemma}
	\label{lem:discrete_time_trans_func_bound}
	Let $A$ be a Riesz-spectral operator on a Hilbert space $X$ whose eigenvalues
	$\{\lambda_n:n \in \mathbb{N} \}$ satisfy
	{\rm (A\ref{assump:finite_unstable})} and $0 
	\in \overline{\{\lambda_n:n \in \mathbb{N} \}} \setminus
	\{\lambda_n:n \in \mathbb{N} \}$.
	Let $B \in \mathcal{L}(X,\mathbb{C})$ and $F \in \mathcal{L}(\mathbb{C},X)$
	be given by \eqref{eq:B_F_rep}, and 
	assume that  
	{\rm (A\ref{assump:b_in_invA})} holds.
	If there exists $\epsilon_{\rm c}\in (0,1)$ such that 
	\begin{equation}
	\label{eq:continuous_lower_bound}
	|1 -  F(\lambda I - A)^{-1}B| > \epsilon_{\rm c} \qquad \forall \lambda \in \rho(A) \cap \overline{\mathbb{C}_0},
	\end{equation}
	then, for every $\epsilon_{\rm d}  \in (0,\epsilon_{\rm c})$,
	there exists $\tau^*>0$ such that
	for every $\tau \in (0,\tau^*)$,
	\begin{equation}
	\label{eq:discrete_lower_bound}
	\big|1 - F\big(z I - T(\tau)\big)^{-1} S(\tau)\big| > \epsilon_{\rm d}\qquad 
	\forall z \in \rho\big(T(\tau)\big) \cap \overline{\mathbb{E}_1}.
	\end{equation} 
\end{lemma}

\begin{proof}
	The proof is divided into four steps.
	In Steps~1 and 2, we study the infinite-dimensional tail
	of the series of $(\lambda I - A)^{-1}B$
	and $(z I - T(\tau))^{-1} S(\tau)$,
	respectively.
	These infinite-dimensional tails
	can be regarded as the errors of 
	finite-dimensional approximations.
	The objective of Steps~1 and 2
	is to show that such an error
	becomes arbitrarily small as the approximation order 
	increases.
	In Step~3, we look at the finite-dimensional approximation
	of the transfer function $F(z I - T(\tau))^{-1} S(\tau)$.
	We do not encounter any difficulties arising from
	strong stability in this step, and so the result on
	exponential stability obtained in the proof 
	of Theorem~2.1 of \cite{Rebarber2006} can be
	applied without modifications.
	In Step~4, we investigate the set 
	$\rho(T(\tau)) \cap \overline{\mathbb{E}_1}$ in order to complete the proof.

	{\em Step 1:}
	We show that 
	for every $\epsilon >0$, there exists 
	$N_0^{\rm c} \in \mathbb{N}$ such that 
	\begin{equation}
	\label{eq:continuous_large_case}
	\sup_{\lambda \in \overline{\mathbb{C}_0} \setminus \{0 \}} 
	\left\|
	\sum_{n = N}^\infty \frac{\langle b , \psi_n \rangle}{\lambda - \lambda_n} \phi_n
	\right\| \leq \epsilon\qquad \forall N\geq N_0^{\rm c}.
	\end{equation}
	
	Let $\lambda \in \overline{\mathbb{C}_0} \setminus \{0 \}$.
	By (A\ref{assump:finite_unstable}), there exists $N_1 \in \mathbb{N}$ such that 
	\begin{equation}
	\label{eq:lambda_cond_large}
	\lambda_n \in \mathbb{C} \setminus \mathbb{C}_{-\alpha}\quad \text{or} \quad 
	\lambda_n \in \mathbb{C} \setminus 
	\Sigma_{\pi/2+\delta}\qquad \forall n \geq N_1.
	\end{equation}
	If $\lambda_n \in \mathbb{C} \setminus  \mathbb{C}_{-\alpha}$, then
	\begin{equation}
	\label{eq:upper_cont_case1}
	\left|
	\frac{1}{\lambda - \lambda_n} + \frac{1}{\lambda_n}
	\right| \leq \frac{2}{\alpha} =: \Gamma_1.
	\end{equation}
	Suppose that $\lambda_n \in \mathbb{C} \setminus  \Sigma_{\pi/2+\delta}$. We obtain
	\[
	\frac{1}{\lambda - \lambda_n} + \frac{1}{\lambda_n}
	= 
	\left(
	1 + \frac{\lambda_n}{\lambda - \lambda_n}
	\right)\frac{1}{\lambda_n}.
	\]
	Since 
	\[
	|\lambda_n| \leq \frac{|\re \lambda_n|}{\sin \delta},\quad
	|\lambda- \lambda_n| \geq |\re \lambda_n|,
	\]
	it follows that 
	\begin{equation}
	\label{eq:upper_cont_case2}
	\left|
	\frac{1}{\lambda - \lambda_n} + \frac{1}{\lambda_n}
	\right| \leq 
	\left(
	1 + \frac{1}{\sin \delta}
	\right)
	\frac{1}{|\lambda_n| } =: \frac{\Gamma_2 }{|\lambda_n|}.
	\end{equation}
	
	Let $\epsilon >0$ be given.
	By
	Lemma~\ref{lem:RB_prop} b) and  (A\ref{assump:b_in_invA}),
	\begin{equation}
	\label{eq:b_bounded}
	\sum_{N=1}^{\infty} \left| 
	\langle b , \psi_n \rangle
	\right|^2 < \infty,\quad
	\sum_{N=1}^{\infty} \left| 
	\frac{\langle b , \psi_n \rangle}{\lambda_n}
	\right|^2 < \infty.
	\end{equation}
	Therefore, there exists $N_2 \geq N_1$ such that
	\[
	\sum_{n=N}^{\infty} \left| 
	\langle b , \psi_n \rangle
	\right|^2 < \frac{\epsilon^2}{8M_{\rm b} \Gamma_1^2},\quad
	\sum_{n=N}^{\infty} \left| 
	\frac{\langle b , \psi_n \rangle}{\lambda_n}
	\right|^2 < \frac{\epsilon^2}{8M_{\rm b} \Gamma_2^2}\qquad \forall N \geq N_2.
	\]
	Combining this with \eqref{eq:RB_constant},
	we obtain
	\[
	\left\|
	\sum_{n=N}^\infty \frac{\langle b ,\psi_n \rangle }{\lambda_n} \phi_n
	\right\|^2 \leq 
	M_{\rm b} \sum_{n=N}^{\infty} \left| 
	\frac{\langle b , \psi_n \rangle}{\lambda_n}
	\right|^2< \frac{\epsilon^2}{4}\qquad \forall N \geq N_2.
	\]
	Moreover, since 
	\eqref{eq:upper_cont_case1} and \eqref{eq:upper_cont_case2} yield
	\[
	\left|
	\frac{1}{\lambda - \lambda_n} + \frac{1}{\lambda_n}
	\right|^2  \leq \Gamma_1^2 + \frac{\Gamma_2^2 }{|\lambda_n|^2}
	\qquad \forall N \geq N_1,
	\]
	it follows from \eqref{eq:RB_constant} that, for every $N \geq N_2$, 
	\begin{align*}
	\left\|
	\sum_{n = N}^\infty \frac{\langle b , \psi_n \rangle}{\lambda - \lambda_n} \phi_n +
	\sum_{n = N}^\infty \frac{\langle b , \psi_n \rangle}{\lambda_n} \phi_n 
	\right\|^2 &\leq M_{\rm b} \sum_{n = N}^\infty \left|
	\frac{1}{\lambda - \lambda_n} + \frac{1}{\lambda_n}
	\right|^2 \cdot |\langle b, \psi_n \rangle  |^2 \\
	&\leq 
	M_{\rm b}  \left( \Gamma_1^2\sum_{n = N}^\infty 
	|\langle b, \psi_n \rangle  |^2 + \Gamma_2^2
	\sum_{n = N}^\infty 
	\frac{|\langle b, \psi_n \rangle  |^2 }{|\lambda_n|^2}
	\right) \\
	&<   \frac{\epsilon^2}{4}.
	\end{align*}
	Therefore,
	\[
	\left\|
	\sum_{n = N}^\infty \frac{\langle b , \psi_n \rangle}{\lambda - \lambda_n} \phi_n 
	\right\| \leq 
	\left\|
	\sum_{n = N}^\infty \frac{\langle b , \psi_n \rangle}{\lambda - \lambda_n} \phi_n +
	\sum_{n = N}^\infty \frac{\langle b , \psi_n \rangle}{\lambda_n} \phi_n 
	\right\| + 
	\left\|
	\sum_{n = N}^\infty \frac{\langle b , \psi_n \rangle}{\lambda_n} \phi_n 
	\right\| < \epsilon
	\]
	for every $N \geq N_2$.
	Thus, 
	\eqref{eq:continuous_large_case} holds with $N_0^{\rm c} := N_2$.

	{\em Step 2:}
	Recall that $(z I - T(t))^{-1}S(t)$
	can be represented in the form \eqref{eq:Resol_S_tau}.
	We shall  show that 
	for every $\epsilon >0$, there exists $N_0^{\rm d} \in \mathbb{N}$
	such that 
	\begin{equation}
	\label{eq:discrete_large_case}
	\sup_{z \in \overline{\mathbb{E}_1} \setminus \{1 \}} 
	\left\|
	\sum_{n = N}^\infty 
	\frac{1 - e^{\tau \lambda_n}}{z - e^{\tau \lambda_n}} \cdot 
	\frac{\langle b , \psi_n \rangle}{\lambda_n} \phi_n
	\right\| \leq  \epsilon\qquad \forall \tau >0 ,~\forall N \geq N_0^{\rm d}.
	\end{equation}
	
	Let $z \in \overline{\mathbb{E}_1} \setminus \{1 \}$ and $\tau >0$.
	As in Step 1, we choose $N_1 \in \mathbb{N}$ so that 
	\eqref{eq:lambda_cond_large} holds.
	The following inequality is useful to obtain the estimate \eqref{eq:discrete_large_case}:
	\begin{equation}
	\label{eq:dis_time_bound}
	\left|
	\frac{1 - e^{\tau \lambda_n}}{z - e^{\tau \lambda_n}}
	\right|
	\leq \frac{|1 - e^{\tau \lambda_n}|}{1 - e^{\tau \re \lambda_n}} = 
	\frac{\frac{|1 - e^{\tau \lambda_n}|}{\tau  |\lambda_n|}}{ \frac{1 - e^{\tau \re \lambda_n}}{\tau |\re \lambda_n|}} \cdot
	\frac{|\lambda_n|}{|\re \lambda_n|}\qquad \forall n \geq N_1.
	\end{equation}

	Let $n \geq N_1$.
	Recalling that \eqref{eq:lambda_cond_large} holds,
	we first consider the case 
	$\lambda_n \in  \mathbb{C} \setminus \mathbb{C}_{-\alpha}$, i.e., 
	$\re \lambda_n \leq -\alpha$.
	Suppose that $-1 \leq \tau \re  \lambda_n \leq 0$. 
		The function
	\[
	g(\lambda) :=
	\begin{cases}
	\frac{1 - e^{\lambda}}{\lambda} & \text{if $\lambda \not=0$} \\
	-1 & \text{if $\lambda = 0$}
	\end{cases}
	\]
	is holomorphic on $\mathbb{C}$. 
	Hence, on a compact set $\{
	\lambda \in \mathbb{C}: -1 \leq \re \lambda \leq 0,~|\im \lambda | \leq \pi 
	\}$, there exists $M_1 >0$ such that $|g(\lambda)| \leq M_1$.
	For every $\lambda \in \mathbb{C}$ with  $|\im \lambda | \leq \pi$, 
	\[
	|g(\lambda \pm 2\ell \pi i)| = \left|\frac{1 - e^{\lambda}}{\re \lambda + i(\im \lambda \pm 2\ell \pi )}\right|
	\leq |g(\lambda)|\qquad \forall \ell \in \mathbb{N}.
	\]
	Therefore, $|g(\lambda)| \leq M_1$ if $-1 \leq \re \lambda \leq 0$.
	This estimate on $g$ shows that
	\begin{equation}
	\label{eq:dis_bount1}
	\frac{|1 - e^{\tau \lambda_n}|}{\tau |\lambda_n|} \leq M_1.
	\end{equation}
	Moreover, by the mean value theorem,
	\begin{equation}
	\label{eq:dis_bount2}
	\frac{1 - e^{\tau \re \lambda_n}}{\tau |\re \lambda_n|} \geq e^{-1}.
	\end{equation}
	It follows from \eqref{eq:dis_time_bound}--\eqref{eq:dis_bount2} that
	\[
	\left|
	\frac{1 - e^{\tau \lambda_n}}{z - e^{\tau \lambda_n}}
	\right| \cdot \left|
	\frac{1}{\lambda_n}
	\right| \leq \frac{e M_1 }{|\re \lambda_n |} \leq \frac{e M_1 }{\alpha}.
	\]
	Suppose that $\tau \re \lambda_n < -1$. Then
	\eqref{eq:dis_time_bound} leads to
	\begin{equation}
	\label{eq:1e_ze_bound}
	\left|
	\frac{1 - e^{\tau \lambda_n}}{z - e^{\tau \lambda_n}}
	\right|  \leq 
	\frac{|1 - e^{\tau \lambda_n}|}{1 - e^{\tau \re \lambda_n}} 
 < \frac{2}{1 - e^{-1}}.
	\end{equation}
	Thus, if $\re \lambda_n \leq -\alpha$, then
	\begin{equation}
	\label{eq:discrete_case_1}
	\left|
	\frac{1 - e^{\tau \lambda_n}}{z - e^{\tau \lambda_n}}
	\right|\cdot \left|
	\frac{1}{\lambda_n}
	\right| \leq 
	\max \left\{
	\frac{e M_1 }{\alpha},\frac{2}{(1 - e^{-1})\alpha}
	\right\} =: \Upsilon_1.
	\end{equation}
	
	Next we consider the case $\lambda_n \in 
	\mathbb{C} \setminus 
	\Sigma_{\pi/2+\delta}$, i.e., $\re \lambda_n \leq 0$ and
	\[
	|\lambda_n| \leq \frac{|\re \lambda_n|}{\sin \delta}.
	\]
	If $-1 \leq \tau \re  \lambda_n \leq 0$, then
	\eqref{eq:dis_time_bound}--\eqref{eq:dis_bount2} yield
	\[
	\left|
	\frac{1 - e^{\tau \lambda_n}}{z - e^{\tau \lambda_n}}
	\right| \leq \frac{eM_1}{\sin \delta}.
	\]
	If $\tau \re \lambda_n < -1$, then
	\eqref{eq:1e_ze_bound} holds.
	Thus, in the case $\lambda_n \in 
	\mathbb{C} \setminus 
	\Sigma_{\pi/2+\delta}$,
	we obtain
	\begin{equation}
	\label{eq:discrete_case_2}
	\left|
	\frac{1 - e^{\tau \lambda_n}}{z - e^{\tau \lambda_n}}
	\right| \leq 
	\max \left\{
	\frac{eM_1}{\sin \delta}, \frac{2}{1 - e^{-1}}
	\right\} =: \Upsilon_2.
	\end{equation}
	
	By the estimates \eqref{eq:discrete_case_1} and \eqref{eq:discrete_case_2}, 
	for every $N \geq N_1$,
	\begin{align}
	\left\|
	\sum_{n = N}^\infty 
	\frac{1 - e^{\tau \lambda_n}}{z - e^{\tau \lambda_n}}
	\cdot
	\frac{\langle b , \psi_n \rangle}{\lambda_n} \phi_n
	\right\|^2 
	&\leq M_{\rm b}
	\sum_{n = N}^\infty 
	\left|
	\frac{1 - e^{\tau \lambda_n}}{z - e^{\tau \lambda_n}}
	\cdot 
	\frac{\langle b , \psi_n \rangle}{\lambda_n}
	\right|^2 \notag \\  
	&\leq M_{\rm b} \left(
	\Upsilon_1^2\sum_{n = N}^\infty  \left|
	\langle b , \psi_n \rangle
	\right|^2  + \Upsilon_2^2 \sum_{n = N}^\infty \left|
	\frac{\langle b , \psi_n \rangle}{\lambda_n}
	\right|^2 
	\right). \label{eq:discrete_time_transfer_suff_large}
	\end{align}
	Similarly to Step 1,
	it follows 
	from \eqref{eq:b_bounded}
	that for every $\epsilon >0$,
	there exists $N_0^{\rm d}  \geq N_1$ such that 
	\eqref{eq:discrete_large_case} holds.
	
	{\em Step 3:}
	By \eqref{eq:Resol_S_tau},
	\[
	1 - F\big(z I - T(\tau)\big)^{-1}S(\tau) = 
	1 + \sum^{\infty}_{n=1}\frac{1-e^{\tau \lambda_n} }{z - e^{\tau \lambda_n}} 
	\cdot 
	\frac{\langle b , \psi_n \rangle \langle \phi_n, f \rangle}{\lambda_n}
	\]
	for all $z \in \rho(T(\tau))$ and $\tau >0$.
	Assume that $\epsilon_{\rm c} \in (0,1)$ satisfies \eqref{eq:continuous_lower_bound}, and
	choose $\epsilon \in (0,\epsilon_{\rm c}/3)$ arbitrarily.
	By Steps~1 and 2,
	there exists $N_0 \in \mathbb{N}$ such that 
	for every $N \geq N_0$ and $\tau >0$,
	\begin{subequations}
		\begin{align}
		\sup_{\lambda \in \overline{\mathbb{C}_0} \setminus \{0 \}} 
		&\left|
		\sum_{n=N}^{\infty} \frac{\langle b , \psi_n \rangle \langle \phi_n, f \rangle}{\lambda - \lambda_n} 
		\right| \leq \epsilon \label{eq:suff_large_cont}\\
		\sup_{z \in \overline{\mathbb{E}_1} \setminus \{1 \}} 
		&\left|
		\sum^{\infty}_{n=N}\frac{1-e^{\tau \lambda_n}}{z - e^{\tau \lambda_n}}
		\cdot  \frac{\langle b , \psi_n \rangle \langle \phi_n, f \rangle}{\lambda_n}
		\right| \leq \epsilon. \label{eq:suff_large_dist}
		\end{align}
	\end{subequations}

	Let
	$N_1 \in \mathbb{N}$ satisfy \eqref{eq:lambda_cond_large}, and
	take $N \geq \max\{N_0,N_1\}$.
	We  investigate the finite-dimensional
	truncation:
	\[
	\sum^{N-1}_{n=1}\frac{1-e^{\tau \lambda_n}}{z - e^{\tau \lambda_n}}
	\cdot  \frac{\langle b , \psi_n \rangle \langle \phi_n, f \rangle}{\lambda_n}.
	\]
	This finite sum has no difficulty arising from strong stability, i.e., 
	$0 \in \sigma(A) \setminus \sigma_p(A)$. Hence
	we can apply the result on exponential stability developed in
	the proof of Theorem~2.1 of
	\cite{Rebarber2006}.
	
	For $\tau, \eta, a>0$, define   
	the sets $\Omega_0$, $\Omega_1$, $\Omega_2$, and
	$\Omega_3$ by
	\begin{align*}
	\Omega_0 &:= 
	\{z = e^{\tau \lambda }: \re \lambda \geq 0,~|\tau \lambda | < \eta\} 
	 \\
	\Omega_1 &:= 
	\{z = e^{\tau \lambda }: |\lambda - \lambda_n| \geq a ~\text{for all 
		$ 1\leq n \leq N-1$}\} \\
	&\qquad \cup
	\{z = e^{\tau \lambda }: 0< |\lambda - \lambda_n| < a,~
	\langle b , \psi_n \rangle \langle \phi_n, f \rangle = 0
	~\text{for some $ 1\leq n \leq N-1$} \} \\
	\Omega_2 &:= \{z = e^{\tau \lambda }: 0< |\lambda - \lambda_n| < a,~
	\langle b , \psi_n \rangle \langle \phi_n, f \rangle \not= 0
	~\text{for some $ 1\leq n \leq N-1$} \} \\
	\Omega_3 &:=\overline{\mathbb{E}_1} \setminus 
	\Omega_0.
	\end{align*}
	If $0< \eta < \pi$, then for every $z \in \Omega_0$,
	there uniquely exists $\lambda \in\overline{\mathbb{C}_0}$ such that  
	$z = e^{\tau \lambda}$  and $|\tau \lambda| < \eta$.
	This $\lambda$ is the complex variable in the continuous-time setting
	corresponding to the complex variable $z$ in the discrete-time setting.

	Define 
	$
	a^* := \min \{|\lambda_n - \lambda_m|/2: 1\leq n<m \leq N-1\}.
	$
	If $|\lambda - \lambda_n| < a^*$ for some $1 \leq n\leq N-1$, then
	$|\lambda - \lambda_m| \geq a^*$ for $1\leq m \leq N-1$ with $m\not=n$.
	By Steps 3) and 4) of the proof of Theorem~2.1 in \cite{Rebarber2006},
	there exist
	$\tau^* >0$, $\eta \in (0,\pi)$, and 
	$a \in (0,a^*)$ 
	such that
	the following three statements hold
	for every $\tau \in (0,\tau^*)$:
	\begin{enumerate}
		\def\theenumi{\roman{enumi}}
		\def\labelenumi{\theenumi}
		\renewcommand{\labelenumi}{(\roman{enumi})}
		\item  \label{it:Omega_4}
		for all $z \in \Omega_4 := \Omega_0 \cap \Omega_1$ and
		the corresponding $\lambda$,
		\begin{equation}
		\label{eq:finite_case1}
		\left|
		\sum_{n=1}^{N-1} \frac{\langle b , \psi_n \rangle \langle \phi_n, f \rangle}{\lambda - \lambda_n}  +
		\sum^{N-1}_{n=1}\frac{1-e^{\tau \lambda_n}}{z - e^{\tau \lambda_n}}
		\cdot  \frac{\langle b , \psi_n \rangle \langle \phi_n, f \rangle}{\lambda_n} 
		\right| < \epsilon;
		\end{equation}
		\item \label{it:Omega_3}
		$e^{\tau \lambda_n} \in \mathbb{C} \setminus \Omega_3$
		for all $1 \leq n \leq N-1$; and 
		\item \label{it:Omega_5}
		for all $z \in \Omega_5 := (\Omega_0 \cap \Omega_2) \cup
		\Omega_3$,
		\begin{equation}
		\label{eq:finite_case2}
		\left|
		1 + \sum^{N-1}_{n=1}\frac{1-e^{\tau \lambda_n}}{z - e^{\tau \lambda_n}}
		\cdot  \frac{\langle b , \psi_n \rangle \langle \phi_n, f \rangle}{\lambda_n} 
		\right|> \epsilon_{\rm c}.
		\end{equation}
	\end{enumerate}
	\def\theenumi{\alph{enumi})}
	\def\labelenumi{\theenumi}
	\def\theenumii{(\roman{enumii})}
	\def\labelenumii{\theenumii}
	In what follows, we set $\tau,\eta, a>0$ so that
	the above statements (\ref{it:Omega_4})--(\ref{it:Omega_5}) hold.

	Suppose that $z \in \Omega_4 \setminus \{1\}$, and let
	$\lambda \in \overline{\mathbb{C}_0} \setminus \{0\}$ be
	the corresponding 
	complex variable in the continuous-time setting.
	Since
	\[
	|1 -  F(\lambda I - A)^{-1}B | =
	\left|
	1 - \sum_{n=1}^{\infty} \frac{\langle b , \psi_n \rangle \langle \phi_n, f \rangle}{\lambda - \lambda_n} 
	\right| > \epsilon_{\rm c},
	\]
	it follows from the estimates
	\eqref{eq:suff_large_cont}, \eqref{eq:suff_large_dist}, and
	\eqref{eq:finite_case1} that 
	\begin{align*}
	\left| 
	1 + \sum_{n=1}^\infty \frac{1-e^{\tau \lambda_n}}{z - e^{\tau \lambda_n}}
	\cdot  \frac{\langle b , \psi_n \rangle \langle \phi_n, f \rangle}{\lambda_n} 
	\right| &> 
	\epsilon_{\rm c} - 3\epsilon.
	\end{align*}
	On the other hand, if $z \in \Omega_5 \setminus \{1\}$, then \eqref{eq:suff_large_dist}
	and 
	\eqref{eq:finite_case2} yield
	\begin{align*}
	\left| 
	1 + \sum_{n=1}^\infty \frac{1-e^{\tau \lambda_n}}{z - e^{\tau \lambda_n}}
	\cdot  \frac{\langle b , \psi_n \rangle \langle \phi_n, f \rangle}{\lambda_n} 
	\right| 
	&> \epsilon_{\rm c}-\epsilon.
	\end{align*}

{\em Step 4:}
It remains to show that
\begin{equation}
\label{eq:Omega45}
(\Omega_4 \setminus \{1 \}) \cup (\Omega_5 \setminus \{1 \})
= 
\rho\big (T(\tau)\big) \cap \overline{\mathbb{E}_1}.
\end{equation}	
	By definition,
	\[
	(\Omega_0 \cap \Omega_1)  \cup 
	(\Omega_0 \cap \Omega_2) = 
	\Omega_0 \cap (\Omega_1 \cup \Omega_2) = 
	\Omega_0 \setminus \{e^{\tau \lambda_n}:  1\leq n \leq N-1  \}.
	\]
	Moreover, the statement (\ref{it:Omega_3}) above yields
	\[
	\Omega_3 \cap \{e^{\tau\lambda_n} : 1 \leq n \leq N-1 \} = \emptyset.
	\]
	Since $N \geq N_1$, it follows from \eqref{eq:lambda_cond_large} that 
	$ \overline{\mathbb{E}_1} \cap 
	\{
	e^{\tau \lambda_n}:  n \geq N 
	\} = \emptyset$.
	Hence
	\begin{align*}
	(\Omega_4 \setminus \{1 \}) \cup (\Omega_5 \setminus \{1 \})  &=
	(\Omega_4  \cup \Omega_5) \setminus \{1 \} \\
	&=
	\big((\Omega_0 \setminus \{e^{\tau \lambda_n}:  1\leq n \leq N - 1  \} )
	\cup \Omega_3\big) \setminus \{1 \} \\
	&= (
	\Omega_0 \cup \Omega_3
	) 
	\setminus ( \{1 \} \cup
	\{e^{\tau\lambda_n} : 1 \leq n \leq N-1 \} )
	\\
	&=
	\overline{\mathbb{E}_1}  \setminus 
	( \{1 \} \cup
	\{e^{\tau\lambda_n} : n \in \mathbb{N} \} ) .
	\end{align*}
	Since  
	$\sigma(T(\tau)) = \overline{\{e^{\tau \lambda_n}:
		n \in \mathbb{N} \}}$ by Lemma~\ref{lem:T_spectrum},
	we obtain
	\[
	\overline{\mathbb{E}_1}  \setminus 
	( \{1 \} \cup
	\{e^{\tau\lambda_n} : n \in \mathbb{N} \} ) =
	\overline{\mathbb{E}_1} \setminus \sigma \big(
	T(\tau)
	\big).
	\]
	Thus, \eqref{eq:Omega45} holds. This completes the proof.
\end{proof}

We are now in the position to prove Theorem~\ref{thm:T_SF_resolvent_set}.
\begin{proof}[Proof of Theorem~\ref{thm:T_SF_resolvent_set}]
	By Lemmas~\ref{lem:continuous_time_trans_func_bound}
	and \ref{lem:discrete_time_trans_func_bound},
	there exists $\tau^*>0$ such that 
	\begin{equation}
	\label{eq:one_FRS_1}
	1 \in  \rho\big(FR\big(z,T(\tau)\big)S(\tau)\big)
	\qquad 
	\forall z \in \rho\big(T(\tau) \big) \cap \overline{\mathbb{E}_1},~\forall \tau \in (0,\tau^*).
	\end{equation}
	Let $\tau \in (0,\tau^*)$.
	By (A\ref{assump:finite_unstable})--(A\ref{assump:origin}) and
	Lemma~\ref{lem:T_spectrum},
	$\sigma (T(\tau)) \cap \mathbb{T} = \{1 \}$. This and \eqref{eq:one_FRS_1} imply that 
	\begin{equation}
	\label{eq:one_FRS_2}
	\mathbb{T} \setminus\{1 \} \subset \rho\big(T(\tau) \big) \cap \overline{\mathbb{E}_1},\qquad
	1 \in  \rho\big(FR\big(z,T(\tau)\big)S(\tau)\big)
	\quad 
	\forall z \in \mathbb{T} \setminus\{1 \} .
	\end{equation}
	On the other hand, 
	\[
	zI - \Delta(\tau) =
	\big(zI - T(\tau)\big) 
	\big(I - \big(zI-T(\tau)\big)^{-1} S(\tau)F\big)\qquad \forall z \in  \rho\big(T(\tau) \big).
	\]
	Since
	$\sigma\big(R(z, T(\tau)) S(\tau)F\big) \setminus \{0 \} = 
	\sigma\big(F R(z, T(\tau))S(\tau)\big) \setminus \{0 \} $
	by Lemma~\ref{lem:bounded_operator_spectrum}, it follows that
	for every $z \in \rho(T(\tau))$,
	\begin{equation}
	\label{eq:resolvent_equivalence}
	1 \in \rho \big(F R\big(z, T(\tau)\big)S(\tau)\big) \quad \Leftrightarrow \quad 
	z \in \rho \big(\Delta(\tau)\big).
	\end{equation}
	By \eqref{eq:one_FRS_2} and \eqref{eq:resolvent_equivalence}, we obtain
	$
	\mathbb{T} \setminus \{1\} \subset \rho(\Delta(\tau))
	$. Since $1 \not\in
	\sigma_p(\Delta(\tau))$  and $1 \in \sigma(\Delta(\tau))$ by
	Lemma~\ref{lem:spectal_cond}, it follows that
	$\sigma_p(\Delta(\tau)) \cap
	\mathbb{T}  = \emptyset $ and
	$\sigma(\Delta(\tau)) \cap
	\mathbb{T}  = \{1\}$.
\end{proof}

\section{Preservation of boundedness}
\label{sec:power_boundedness}
In this section, we  prove the power boundedness of 
$(\Delta(\tau)^k)_{k \in \mathbb{N}}$ in order to finish the proof of the 
main theorem.
\begin{theorem}
	\label{thm:power_boundedness}
	If Assumption~\ref{assum:for_MR} is satisfied, then
	there exists $\tau^* >0$ 
	such that 
	the discrete semigroup $(\Delta(\tau)^k)_{k \in \mathbb{N}}$
	is power bounded for every $\tau \in (0,\tau^*)$.
\end{theorem}	
\begin{remark}
	\label{rem:discrete_cont_spec}
	{\em 
	Combining Theorems \ref{thm:T_SF_resolvent_set} and 
	\ref{thm:power_boundedness} with
	the discrete version of the mean ergodic theorem 
	(see, e.g., Theorem~I.2.9 and 
	Corollary~I.2.11 in \cite{Eisner2010}),
	we obtain 
	\[
	X = \ker (\Delta(\tau) - 1) \oplus \overline{\ran (\Delta(\tau) - 1)}=
	\overline{\ran (\Delta(\tau) - 1)}
	\]
	for all sufficiently small $\tau>0$. Therefore, $1$ belongs to
	the continuous spectrum of $\Delta(\tau)$.
}
\end{remark}

Before proving Theorem~\ref{thm:power_boundedness}, 
we apply a spectral decomposition for $A$;
see, e.g.,
Lemma 2.4.7 of \cite{Curtain2020} or Proposition IV.1.16 in \cite{Engel2000}.
Assume that 
(A\ref{assump:finite_unstable}) is satisfied. Then
only finite elements of $\{\lambda_n :n \in \mathbb{N}\}$ are in
$\mathbb{C}_{-\alpha} \cap \Sigma_{\pi/2+\delta}$.
For every $\beta >0$,
there exists  a smooth, positively oriented, and simple closed curve $\Phi$ in 
$\rho(A)$ containing $\sigma(A) \cap \overline{\mathbb{C}_\beta}$ in its interior and 
$\sigma(A) \cap (\mathbb{C} \setminus \overline{\mathbb{C}_\beta})$ in its exterior. 
Here we choose $\beta>0$ so that $\sigma(A) \cap \overline{\mathbb{C}_\beta} = 
\{\lambda_n:n \in \mathbb{N} \} \cap \mathbb{C}_0$.
The operator
\begin{equation}
\label{eq:projection}
\Pi := \frac{1}{2\pi i} \int_{\Phi} (\lambda I - A)^{-1} d\lambda
\end{equation}
is a projection on $X$. We have
\[
X = X^+ \oplus X^-,
\]
where $X^+ := \Pi X$ and $X^- := (I - \Pi) X$.
Then $\dim X^+ < \infty$, $X^+$ and $X^-$ are 
$T(t)$-invariant for all $t \geq 0$, and
\[
\sigma(A^+) = \sigma(A) \cap \overline{\mathbb{C}_\beta},\quad
\sigma(A^-) = \sigma(A) \cap (\mathbb{C} \setminus \overline{\mathbb{C}_\beta}),
\]
where $A^+ := A|_{X^+}$ and $A^- := A|_{D(A) \cap X^-}$.
For $t \geq 0$, we define
\[
T^+(t) := T(t)|_{X^+},\quad T^-(t) := T(t)|_{X^-}.
\]
Then $(T^+(t))_{t\geq 0}$ 
and $(T^-(t))_{t\geq 0}$ are strongly continuous semigroups 
on $X^+$ and $X^-$, respectively, and their
generators  are given by $A^+$ and $A^-$,
respectively.
Let 
\begin{equation}
\label{eq:Ns_def}
\{\lambda_n:1\leq n \leq N_{\rm s} - 1 \}  := \{\lambda_n 
: n \in \mathbb{N}\} \cap \mathbb{C}_0
= \sigma(A^+)
\end{equation}
 by changing the order of $\{\lambda_n 
: n \in \mathbb{N}\}$ if necessary.

By construction, $(T^-(t))_{t\geq 0}$ is uniformly bounded.
Moreover, Lemma~4.2.7 of \cite{Curtain2020} shows that 
$(T^-(t))_{t\geq 0}$ is strongly stable under (A\ref{assump:imaginary}).
Hence we easily obtain the following properties of
the discrete semigroup $(T^{-}(\tau)^k)_{k \in\mathbb{N}}$.
\begin{lemma}
	\label{lem:T_minus_strong_stable}
	Let $\tau > 0$.
	The discrete semigroup $(T^{-}(\tau)^k)_{k \in\mathbb{N}}$
	constructed as above
	is power bounded.
	If {\em (A\ref{assump:imaginary})} 
	additionally holds, then 
	$(T^{-}(\tau)^k)_{k \in\mathbb{N}}$
	is strongly stable.
\end{lemma}

\begin{remark}
	\label{rem:continuous_spectrum}
	{\em
	Applying the mean ergodic theorem (see, e.g., Theorem~I.2.25 of \cite{Eisner2010}) to 
	the uniformly bounded semigroup $(T^-(t))_{t\geq 0}$,
	we obtain
	\[
	X^- = \ker (A^-) \oplus \overline{\ran (A^-)} = \overline{\ran (A^-)}.
	\]
	Since the finite-dimensional unstable part $A^+$ is invertible, it follows that $X^+ = \ran (A^+)$. Thus, 
	$X = \overline{\ran (A)}$, which implies that $0$ belongs to the continuous spectrum of $A$.
	}
\end{remark}

A characterization of uniformly 
bounded semigroups on Hilbert spaces has been
obtained in \cite{Gomilko1999, Shi2000}.
The following theorem is 
its discrete analogue and
provides a necessary and sufficient condition for
a discrete semigroup on a Hilbert space to be power bounded.
\begin{theorem}[Theorem~II.1.12 of \cite{Eisner2010}]
	\label{thm:power_bounded}
	Let $X$ be a Hilbert space and  $\Delta \in \mathcal{L}(X)$ satisfy $\mathbb{E}_1 \subset \rho(\Delta)$.	
	The discrete semigroup $(\Delta^k)_{k \in \mathbb{N}}$ is 
	power bounded if and only if for every $x,y \in X$,
	\[
	\limsup_{r \downarrow 1} ~(r-1)
	\int^{2\pi}_0 \big(
	\|R(re^{i\theta}, \Delta)x\|^2 + \|R(re^{i\theta},\Delta)^* y\|^2 
	\big) d\theta < \infty.
	\]
\end{theorem}

To use Theorem~\ref{thm:power_bounded},
we show that $\mathbb{E}_1 \subset \rho (\Delta(\tau))$  holds for all sufficiently small $\tau >0$.
\begin{lemma}
	\label{lem:resol_T_SF}
	Let $A$ be a Riesz-spectral operator on a Hilbert space $X$ whose eigenvalues
	$\{\lambda_n:n \in \mathbb{N} \}$ satisfy
	{\rm (A\ref{assump:finite_unstable})} and $0 
	\in \overline{\{\lambda_n:n \in \mathbb{N} \}} \setminus
	\{\lambda_n:n \in \mathbb{N} \}$.
	Assume that $B \in \mathcal{L}(X,\mathbb{C})$ and
	$F \in \mathcal{L}(\mathbb{C},X)$ in the form of \eqref{eq:B_F_rep}
	satisfy
	{\rm (A\ref{assump:closed_loop})--(A\ref{assump:b_F_cond})}. 
	Then there exists $\tau^* >0$ 
	such that 
	$\mathbb{E}_1 \subset \rho (\Delta(\tau))$ 
	for every $\tau \in (0,\tau^*)$.
\end{lemma}
\begin{proof}
	Lemmas~\ref{lem:continuous_time_trans_func_bound} and
	\ref{lem:discrete_time_trans_func_bound} show that
	there exist $\epsilon >0$ and $\tau^* >0$ such that 
	for every $\tau \in (0,\tau^*)$,
	\begin{align}
	&\tau (\lambda_n - \lambda_m) \not= 2\ell \pi i \qquad 
	\forall \ell \in \mathbb{Z} \setminus \{0\} ,~1 \leq n,m \leq N_{\textrm s}-1
	\label{eq:sampling_cond}
	\\
	\label{eq:1_FRS}
	&\big|1 - F R\big(z, T(\tau)\big)S(\tau)\big| 
	>\epsilon \qquad \forall z \in \rho\big(T(\tau)\big) \cap \overline{\mathbb{E}_1},
	\end{align}
where $ N_{\textrm s} \in \mathbb{N}$ is as given in \eqref{eq:Ns_def}.
	Let $\tau \in (0,\tau^*)$.
	Since
	\[
	zI - \Delta(\tau) =
	\big(zI - T(\tau)\big) 
	\big(I - \big(zI-T(\tau)\big)^{-1} S(\tau)F\big)\qquad \forall z \in  \rho\big(T(\tau) \big),
	\]
	it follows from Lemma~\ref{lem:bounded_operator_spectrum} that
	for every $z \in \rho(T(\tau))$,
	\begin{equation}
	1 \in \rho \big(F R\big(z, T(\tau)\big)S(\tau)\big) \quad \Leftrightarrow \quad 
	z \in \rho \big(\Delta(\tau)\big).
	\end{equation}
	Combining this with 
	\eqref{eq:1_FRS}, we obtain
	\begin{equation}
	\label{eq:T_SF_resol}
	\rho\big(T(\tau)\big) \cap \overline{\mathbb{E}_1} \subset 
	\rho\big(\Delta(\tau)\big).
	\end{equation}
	We see from \eqref{eq:T_SF_resol} that
	if $\sigma(T(\tau)) \cap \mathbb{E}_1 \subset \rho(\Delta(\tau))$,
	then
	the desired conclusion $\mathbb{E}_1 \subset \rho (\Delta(\tau))$ holds.
	Assume, to get a contradiction, that 
	$\sigma(T(\tau)) \cap \mathbb{E}_1  \cap 
	\sigma(\Delta(\tau)) \not=
	\emptyset$, and let
	$z_0 \in 
	\sigma(T(\tau)) \cap \mathbb{E}_1  \cap 
	\sigma(\Delta(\tau))$. 
	By 
	Lemma~\ref{lem:T_spectrum}, 
	\[
	\sigma(T(\tau)) \cap \mathbb{E}_{1} = 
	\{e^{\tau \lambda_n} : 1 \leq n \leq N_{\rm s}-1\},
	\]
	and 
	\eqref{eq:sampling_cond} yields
	$e^{\tau \lambda_n} \not= e^{\tau \lambda_m}$ for
	$1 \leq n,m \leq N_{\rm s}-1$ with $n \not=m$.
	Therefore, there uniquely exists $1 \leq n_0 \leq N_{\rm s}-1$
	such that $z_0 = e^{\tau\lambda_{n_0}}$.
	Since $\Delta(\tau) - T(\tau) = S(\tau)F$ is compact,
	it follows that $z_0$ is an eigenvalue of $\Delta(\tau)$.

	Let $v \in X$ be an eigenvector of $\Delta(\tau)$ corresponding
	to  $z_0$. By \eqref{eq:S_rep},
	\begin{equation}
	\label{eq:z_0_eigvec}
	T(\tau) v+ A^{-1}(T(\tau) - I) BFv = z_0v.
	\end{equation}
	Since
	\begin{align*}
	\langle T(\tau) v, \psi_{n_0} \rangle
	&= e^{\tau \lambda_{n_0}} \langle  v, \psi_{n_0} \rangle \\
	\langle A^{-1}(T(\tau) - I) BFv, \psi_{n_0} \rangle &=
	\frac{e^{\tau \lambda_{n_0}}-1}{\lambda_{n_0}} \langle
	BFv, \psi_{n_0}
	\rangle,
	\end{align*}
	it follows from \eqref{eq:z_0_eigvec} that 
	\[
	\frac{e^{\tau \lambda_{n_0}}-1}{\lambda_{n_0}} \langle
	BFv, \psi_{n_0}
	\rangle = 0.
	\]
	By $\lambda_{n_0} \in \mathbb{C}_0$, we obtain
	$\langle
	BFv, \psi_{n_0}
	\rangle = 0$.
	This occurs if and only if 
	$\langle b, \psi_{n_0} \rangle = 0$ or $Fv = 0$.
		Since 	
	$A+BF$ generates a uniformly bounded semigroup
	by (A\ref{assump:closed_loop}), it follows that 
	$(A,B,-)$ is $\beta$-exponentially 
	stabilizable for every $\beta>0$.
	Hence $\langle b, \psi_{n_0} \rangle \not= 0$ by
	Theorem~8.2.3 of \cite{Curtain2020}, and so 
	$Fv = 0$. Substituting it into \eqref{eq:z_0_eigvec}, we obtain 
	$T(\tau) v = e^{\tau \lambda_{n_0}} v$.
	There exists a nonzero constant $\gamma \in \mathbb{C}$
	such that 
	$v = \gamma \phi_{n_0}$.
	Therefore,
	\[
	(A+BF)v = Av = \lambda_{n_0} v,
	\]
	which implies that $\lambda_{n_0} \in \sigma_p(A+BF)$ and
	contradicts the assumption that $A+BF$ generates a uniformly bounded semigroup. This completes the proof.	
\end{proof}

To study power boundedness based on Theorem~\ref{thm:power_bounded}, we use
the well-known Sherman-Morrison-Woodbury formula
given in the next lemma, which can be obtained from a straightforward calculation.
\begin{proposition}
	\label{Prop:SMW}
	Let $X,U$ be Banach spaces, $A:D(A)\subset X \to X$ be a closed operator,
	$B \in \mathcal{L}(U,X)$, $F \in \mathcal{L}(X,U)$, and $\lambda \in \rho(A)$.
	If $1 \in \rho (FR(\lambda,A)B) $, then $\lambda \in \rho (A+BF)$ and 
	\[
	R(\lambda,A+BF) = R(\lambda,A) + R(\lambda,A) B 
	(I-FR(\lambda,A)B)^{-1} FR(\lambda,A).
	\]
\end{proposition}

After these preparations, we are now ready to prove that the discrete semigroup
$(\Delta(\tau)^k)_{k \in \mathbb{N}}$ is power bounded for
all sufficiently small $\tau >0$. The proof is inspired by Paunonen's proof 
of Theorem~4 in \cite{Paunonen2014JDE}.
\begin{proof}[Proof of Theorem~\ref{thm:power_boundedness}]
	By Lemmas~\ref{lem:continuous_time_trans_func_bound},
	\ref{lem:discrete_time_trans_func_bound},
	and \ref{lem:resol_T_SF},
	there exist $\tau^*  >0$ and  $M_0>0$ such that 
	for every $\tau \in (0,\tau^*)$, we obtain $\mathbb{E}_1 \subset \rho (\Delta(\tau))$ and 
	\begin{align*}
	&\left|\frac{1}{1 - FR\big(z, T(\tau)\big) S(\tau) } \right| \leq M_0\qquad 
	\forall z \in \rho \big(T(\tau)\big) \cap \overline{\mathbb{E}_1}.
	\end{align*}
	Let $\tau \in (0,\tau^*)$ be given.
	By Theorem~\ref{thm:power_bounded}, it suffices to show that
	\begin{equation}
	\label{eq:resol_estimate_for_PB}
	\limsup_{r \downarrow 1} ~(r-1)
	\int^{2\pi}_0 \big(
	\big\|R\big(re^{i\theta}, \Delta(\tau)\big)x\big\|^2 + \big\|R\big(re^{i\theta},\Delta(\tau)\big)^* y\big\|^2 
	\big) d\theta < \infty
	\end{equation} 
	for every $x,y \in X$.
	Since
	$\sigma(T(\tau)) = \overline{\{e^{\tau \lambda_n} :
		n \in \mathbb{N}\}}$
	by 	Lemma~\ref{lem:T_spectrum},
	it follows from
	(A\ref{assump:finite_unstable}) that
	there exists $r_1 >1$ such that
	$re^{i\theta} \in \rho(T(\tau))$ 
	for every $r \in (1,r_1)$ and every $\theta \in [0,2\pi)$.
	Since
	the Sermann-Morrison-Woodbury formula given in Proposition~\ref{Prop:SMW} 
	yields
	\[
	R(re^{i\theta}, T(\tau)+S(\tau)F) x =
	R\big(re^{i\theta}, T(\tau)\big) x +
	\frac{R\big(re^{i\theta}, T(\tau)\big) S(\tau)FR\big(re^{i\theta}, T(\tau)\big)  x }
	{1 - FR\big(re^{i\theta}, T(\tau)\big) S(\tau) }
	\]
	for all $x \in X$ and all $r \in (1,r_1)$,
	we can estimate
	\begin{align}
	&\int^{2\pi}_0 
	\|R(re^{i\theta}, T(\tau)+S(\tau)F) x\|^2 d\theta \notag \\
	&\hspace{40pt} \leq
	2 \int^{2\pi}_0 
	\big\|R\big(re^{i\theta}, T(\tau)\big) x\big\|^2 d\theta 		\label{eq:R_square_int}\\
	&\hspace{60pt}  + 2M_0^2 \|x\|^2 \int^{2\pi}_0 \big\|R\big(re^{i\theta}, T(\tau)\big) S(\tau)\big\|^2 
	\cdot
	\big\|FR\big(re^{i\theta}, T(\tau)\big)\big\|^2 d\theta \notag 
	\end{align}
	for all $x \in X$ and all $r \in (1,r_1)$.

	To estimate the first term of the 
	right-hand side of \eqref{eq:R_square_int}, we
	apply the spectral decomposition by the projection $\Pi$
	given in \eqref{eq:projection}. 
	Take $x \in X$. Then $x = x^+ + x^-$ with $x^+:= \Pi x \in X^+$ and 
	$x^- := (I -\Pi)x \in X^-$.
	There exists $c_1 >0$ such that $|r e^{i \theta} - e^{\tau \lambda_n} | \geq c_1$ for all $r \in (1,r_1)$,
	$\theta \in [0,2\pi)$, and
	$1 \leq n \leq  N_{\rm s}-1$. Therefore,
	 \eqref{eq:RB_constant} and Lemma~\ref{lem:T_spectrum} yield
	\begin{align*}
	\int^{2\pi}_0 
	\big\|R\big(r e^{i \theta}, T(\tau)\big) x^+\big\|^2 d\theta
	&\leq M_{\rm b} \sum_{n=1}^{N_{\rm s}-1} |\langle 
	x^+ , \psi_n
	\rangle|^2 \int^{2\pi}_0 \frac{1}{|re^{i\theta} - e^{\tau \lambda_n} |^2 } d\theta \\
	&\leq \frac{2\pi M_{\rm b}}{ c_1^2} \sum_{n=1}^{N_{\rm s}-1} |\langle 
	x^+ , \psi_n
	\rangle|^2 .
	\end{align*}
	Therefore,
	\[
	\limsup_{r\downarrow 1} ~(r-1)
	\int^{2\pi}_0 
	\big\|R\big(r e^{i \theta}, T(\tau)\big) x^+\big\|^2  d\theta = 0.
	\]
	Since 
	the discrete semigroup
	$(T^-(\tau)^k )_{k \in \mathbb{N}}$ is power bounded 
	by Lemma~\ref{lem:T_minus_strong_stable},  we see from
	Theorem~\ref{thm:power_bounded} that
	\[
	\limsup_{r\downarrow 1} ~(r-1)
	\int^{2\pi}_0 
	\big\|R\big(r e^{i \theta}, T(\tau)\big) x^-\big\|^2  d\theta < \infty.
	\]
	Consequently,
	\begin{equation}
	\label{eq:Resol_T_int}
	\limsup_{r\downarrow 1} ~(r-1)
	\int^{2\pi}_0 
	\big\|R\big(re^{i\theta}, T(\tau)\big) x\big\|^2 d\theta < \infty.
	\end{equation}
	
	We next investigate 
	the second term of the right-hand side of \eqref{eq:R_square_int}.
	Using \eqref{eq:RB_constant} and 
	\eqref{eq:Resol_S_tau}, we have that for every 
	$r e^{i\theta} \in \rho(T(\tau))$,
	\[
	\big\|R\big(re^{i\theta}, T(\tau)\big) S(\tau)\big\|^2 \leq 
	M_{\rm b} \sum_{n=1}^{\infty} 
	\left|
	\frac{1 - e^{\tau \lambda_n}}{r e^{i\theta} - e^{\tau \lambda_n}}
	\right|^2 \cdot  \left| \frac{\langle b, \psi_n \rangle }{\lambda_n}\right|^2 .
	\]
	Under (A\ref{assump:finite_unstable}),
	there is $N_1 \in \mathbb{N}$ such that
	$\lambda_n \in \mathbb{C} \setminus \mathbb{C}_{-\alpha}$ or
	$\lambda_n \in \mathbb{C} \setminus  
	\Sigma_{\pi/2+\delta}$ for every $N \geq N_1$.
	As shown in \eqref{eq:discrete_time_transfer_suff_large}
	in Step 2 of the proof of Lemma~\ref{lem:discrete_time_trans_func_bound},
	there exists $M_1 >0$ such that
	for every $z \in \overline{\mathbb{E}_1} \setminus \{1 \}$,
	\[
	\sum_{n=N_1}^{\infty} 
	\left|
	\frac{1 - e^{\tau \lambda_n}}{z - e^{\tau \lambda_n}}
	\right|^2\cdot  \left| \frac{\langle b, \psi_n \rangle }{\lambda_n}\right|^2  \leq M_1.
	\]
	It follows from (A\ref{assump:imaginary}) that
	there exist 
	$c_2>0$ and $M_2 >0$ such that
	for all $1\leq n \leq N_1-1$, 
	\begin{align*}
	&|re^{i\theta} - e^{\tau \lambda_n}| \geq c_2\qquad \forall r \in (1,r_1),~ \forall \theta \in [0,2\pi) \\
	&|1-e^{\tau \lambda_n}| \leq 1+|e^{\tau \lambda_n}| \leq M_2.
	\end{align*}
	By these inequalities,
	\[
	\sum_{n=1}^{N_1 -1} 
	\left|
	\frac{1 - e^{\tau \lambda_n}}{r e^{i\theta} - e^{\tau \lambda_n}}
	\right|^2 \cdot \left| \frac{\langle b, \psi_n \rangle }{\lambda_n}\right|^2 \leq 
	\left(\frac{M_2}{c_2} \right)^2
	\sum_{n=1}^{N_1-1} \left| \frac{\langle b, \psi_n \rangle }{\lambda_n}\right|^2 ~~\quad 
	\forall r \in (1,r_1),~\forall \theta \in [0,2\pi).
	\]
	Hence we obtain
	\begin{equation}
	\label{eq:R_S_bound}
	\big\|
	R\big(re^{i\theta}, T(\tau)\big) S(\tau)\big\|^2 \leq M_3\qquad \forall r\in (1,r_1),~\forall \theta \in [0, 2\pi)
	\end{equation}
	for some $M_3 > 0$.
	
	Using the estimate \eqref{eq:R_S_bound}, we have that
	for every $r\in (1,r_1)$,
	\begin{align}
	\label{eq:RS_FR_bound}
	\int^{2\pi}_0 \|R(re^{i\theta}, T(\tau)) S(\tau)\|^2 \cdot  \|FR(re^{i\theta}, T(\tau))\|^2 d\theta 
	&\leq 
	M_3 \int^{2\pi}_0  \|R(re^{i\theta}, T^*(\tau)) F^*\|^2 d\theta .
	\end{align}
	The adjoint semigroup $(T^*(t) )_{t\geq 0}$ is given by
	\[
	T^*(t) x = \sum_{n=1}^\infty e^{t \overline{\lambda_n}} \langle x ,\phi_n \rangle \psi_n\qquad \forall x \in X,~\forall t \geq 0,
	\]
	and its generator is $A^*$; see, e.g., Theorem 2.3.6 of \cite{Curtain2020}.
		Define the operator $A_1:
		D(A_1) \subset X \to X$ by
	\[
	A_1 x := \sum_{n=1}^\infty \overline{\lambda_n} \langle 
	x, \phi_n
	\rangle \psi_n 
	\]
	with domain
	\[
	D(A_1) := \left\{
	x \in X :
	\sum_{n=1}^\infty 
	\left|
	\lambda_n
	\right|^2 \cdot \left|
	\langle x, \phi_n \rangle
	\right|^2 < \infty
	\right\}.
	\]
	By Corollary 3.2.10 of \cite{Curtain2020}, $A_1$
	is a Riesz-spectral operator and generates the semigroup 
	$(T^*(t) )_{t\geq 0}$. Therefore,
	$A_1 = A^*$.
	Similarly to \eqref{eq:Resol_T_int}, we obtain
	\begin{equation}
	\label{eq:T_adjoint_bound}
	\limsup_{r\downarrow 1} ~(r-1)
	\int^{2\pi}_0 
	\big\|R\big(re^{i\theta}, T^*(\tau)\big) y
	\big\|^2 d\theta < \infty\qquad \forall y \in X.
	\end{equation}
	Since 
	$F^*u = fu$ for every $u \in \mathbb{C}$, it follows from
	\eqref{eq:RS_FR_bound} and
	\eqref{eq:T_adjoint_bound} that 
	\begin{equation}
	\label{eq:RSFR_bound}
	\limsup_{r\downarrow 1} ~(r-1)
	\int^{2\pi}_0 \big\|R\big(re^{i\theta}, T(\tau)\big) S(\tau)\big\|^2 \cdot
	\big\|FR\big(re^{i\theta}, T(\tau)\big)\big\|^2 d\theta < \infty.
	\end{equation}
	
	Applying the estimates \eqref{eq:Resol_T_int} and \eqref{eq:RSFR_bound}
	to \eqref{eq:R_square_int},
	we obtain
	\begin{align*}
	\limsup_{r\downarrow 1} ~(r-1)
	\int^{2\pi}_0 
	\big\|R\big(re^{i\theta}, \Delta(\tau)\big) x\big\|^2 d\theta < \infty.
	\end{align*}
	We have from
	a similar calculation  that
	\begin{align*}
	\limsup_{r\downarrow 1} ~(r-1)
	\int^{2\pi}_0 
	\big\|R\big(re^{i\theta}, \Delta(\tau)\big)^* y\big\|^2 d\theta < \infty\qquad \forall y\in X,
	\end{align*}
	using the following estimate:
	\begin{align*}
	&\int^{2\pi}_0 
	\|R(re^{i\theta}, T(\tau)+S(\tau)F)^* y\|^2 d\theta  \\
	&\quad =
	\int^{2\pi}_0 
	\left\|R\big(re^{i\theta}, T(\tau)\big)^* y +
	\left[ \frac{R\big(re^{i\theta}, T(\tau)\big) 
		S(\tau)FR\big(re^{i\theta}, T(\tau)\big) }
	{1 - FR\big(re^{i\theta}, T(\tau)\big) S(\tau) F} \right]^* y \right\|^2 d\theta 
	\\
	&\quad \leq
	2 \int^{2\pi}_0 
	\big\|R\big(re^{i\theta}, T^*(\tau)\big) y\big\|^2 d\theta \\
	&\qquad + 2M_0^2 \|y\|^2 \int^{2\pi}_0 
	\big\|R\big(re^{i\theta}, T(\tau)\big) S(\tau)\big\|^2 \cdot 
	\big\|FR\big(re^{i\theta}, T(\tau)\big)\big\|^2 d\theta
	\end{align*}
	for all $y \in X$ and all $r \in (1,r_1)$.
	Thus, 
	the desired estimate \eqref{eq:resol_estimate_for_PB} is
	obtained for every $x,y \in X$.
\end{proof}

We see from Theorems
\ref{thm:T_SF_resolvent_set} and \ref{thm:power_boundedness} that
the sufficient condition for strong stability in the Arendt-Batty-Lyubich-V\~u theorem
is satisfied. We finally prove the main theorem of this article, Theorem~\ref{thm:SD_SS}.
\begin{proof}[Proof of Theorem~\ref{thm:SD_SS}]
	There exists $\tau^*>0$ such that 
	for every $\tau \in (0,\tau^*)$,
	\begin{enumerate}
		\def\theenumi{\roman{enumi}}
		\def\labelenumi{\theenumi}
		\renewcommand{\labelenumi}{(\roman{enumi})}
		\item 
		$\sigma_p(\Delta(\tau)) \cap
		\mathbb{T}  = \emptyset$ and
		$\sigma(\Delta(\tau)) \cap \mathbb{T} = \{1 \}$
		by Theorem~\ref{thm:T_SF_resolvent_set}; and
		\item the discrete semigroup 
		$(\Delta(\tau)^k)_{k \in \mathbb{N}}$
		is power bounded by Theorem~\ref{thm:power_boundedness}.
	\end{enumerate}
	\def\theenumi{\alph{enumi})}
	\def\labelenumi{\theenumi}
	\def\theenumii{(\roman{enumii})}
	\def\labelenumii{\theenumii}
	By the Arendt-Batty-Lyubich-V\~u theorem, Theorem~\ref{thm:ABLV_disc}, 
	$(\Delta(\tau)^k)_{k \in \mathbb{N}}$ is strongly stable.
	This and Proposition~\ref{prop:SD_DT} show that 
	the sampled-data system \eqref{eq:sampled_data_sys} is strongly stable.
\end{proof}

We conclude this section by applying Theorem~\ref{thm:SD_SS}
to an infinite-dimensional system whose generator is a simple
diagonal operator.
\begin{example}
	\label{ex:diagonal_case}
	{\em
		Let $X = \ell^2(\mathbb{C})$ with 
		standard basis $\{\phi_n : n \in \mathbb{N}\}$, $N_{\rm s} \in \mathbb{N}$, and 
		$\{\lambda_n \in \mathbb{C}_0: 
		1 \leq n \leq N_{\rm s}-1\}$ be distinct. Define $A \in \mathcal{L}(X)$ by
		\[
		Ax :=  \sum_{n=1}^{N_{\rm s}-1} \lambda_n \langle x,\phi_n \rangle \phi_n + 
		\sum_{n=N_{\rm s}}^{\infty} -\frac{1}{n}\langle x,\phi_n \rangle \phi_n.
		\]
		The operator $A$ satisfies
		(A\ref{assump:finite_unstable})--(A\ref{assump:origin}). 
		Let 
		$b \in X$ and 
		the control operator $B \in \mathcal{L}(\mathbb{C},X)$ be represented as
		$Bu = bu$ for $u \in \mathbb{C}$.
		We apply the spectral decomposition by the projection $\Pi$
		given in \eqref{eq:projection}, and we
		define $B^+ := \Pi B$ and $B^- := (I-\Pi) B$.
		Suppose that $b$ satisfies 
		\begin{equation}
		\label{eq:b_cond_ex}
		\langle b,\phi_n \rangle \not= 0,\quad 1 \leq n \leq N_{\rm s} - 1;\qquad 
		\sum_{n=N_{\rm s}}^{\infty} n^2 |\langle b,\phi_n \rangle |^2 < \infty.
		\end{equation}
		These conditions are equivalent to
		the controllability of the unstable part $(A^+,B^+)$ and (A\ref{assump:b_in_invA}), respectively.
		
		Since $(A^+,B^+)$ is controllable,
		there exists
		$f_1 \in X$ such that the matrix
		\begin{align*} 
		\begin{bmatrix}
		\lambda_1 & & 0 \\
		& \ddots &  \\
		0 & & \lambda_{N_{\rm s}-1}
		\end{bmatrix} + 
		\begin{bmatrix}
		\langle b, \phi_1 \rangle \\
		\vdots \\
		\langle b, \phi_{N_{\rm s}-1} \rangle
		\end{bmatrix}
		\begin{bmatrix}
		\langle \phi_1, f_1 \rangle & \cdots & 
		\langle \phi_{N_{\rm s} - 1}, f_1 \rangle
		\end{bmatrix}
		\end{align*}		 
		is Hurwitz and 
		\begin{align*} 
		\langle \phi_n,f_1 \rangle = 0\qquad \forall n \geq N_{\rm s}.
		\end{align*}
		Let $F_1 \in \mathcal{L}(X,\mathbb{C})$ be represented as $F_1x = \langle x, f_1 \rangle$ for $x \in X$,
		and define $F_1^+ := F_1|_{X^+}$.
		Then $\rho(A^++B^+F_1^+) \supset \overline{\mathbb{C}_0}$.
		For every $\lambda \in \rho(A^++B^+F_1^+) \cap \rho(A^-)$, we obtain 
		$\lambda \in \rho(A+BF_1)$ and write
		\begin{align}
		\label{eq:R_ABF_resol_ex}
		R(\lambda, A+BF_1) = 
		\begin{bmatrix}
		R(\lambda, A^++B^+F_1^+) & 0 \\
		R(\lambda,A^-)B^-F_1^+ R(\lambda, A^++B^+F_1^+) & R(\lambda,A^-)
		\end{bmatrix}
		\end{align}
		under the decomposition $X = X^+ \oplus X^-$. Moreover,
		we see from Theorem~\ref{thm:RS_C0}~a) that
		\begin{equation}
		\label{eq:A_minus_resol_ex}
		\|R(i\omega,A^{-})\| = \frac{1}{|\omega|}.
		\end{equation}
		By  \eqref{eq:R_ABF_resol_ex} and \eqref{eq:A_minus_resol_ex},
		the generator $\widetilde A := A+BF_1$  satisfies
		the condition \eqref{eq:imaginary_axis_RS_case} in 
		(A\ref{assump:closed_loop}). It is not difficult to see that 
		the other conditions in (A\ref{assump:closed_loop}) are also satisfied.
		We find that 
		the feedback operator $F_1$ satisfies (A\ref{assump:b_F_cond}), by
		adding a small perturbation if needed.
		Therefore, we can apply Theorem~\ref{thm:SD_SS} to the sampled-data system
		with the feedback operator $F_1$.
		However, the discrete-time counterpart of Lemma~20 in \cite{Hamalainen2010} immediately shows
		that the structured feedback operator $F_1$ achieves the strong stability of
		the sampled-data system. 
		To illustrate the effectiveness of Theorem~\ref{thm:SD_SS}, we here consider feedback operators
		that affect the stable part $(A^-, B^-)$.
		
		By \eqref{eq:R_ABF_resol_ex} and \eqref{eq:A_minus_resol_ex}, we obtain
		\[
		\sup_{0 < |\omega| \leq 1} |\omega| \cdot
		\| R(i\omega, \widetilde A) \| < \infty.
		\]
		Furthermore, $b \in  D(\widetilde A^{-1})$ holds. Indeed,
		since $b^- := (I-\Pi)b \in D((A^-)^{-1})$ by the latter condition on $b$
		given in \eqref{eq:b_cond_ex},
		there exists $x_b^- \in X^-$ such that 
		$b^-  = A^- x_b^-$.	
		We obtain 
		\[
		\widetilde A\begin{bmatrix}
		x^+ \\ x^-
		\end{bmatrix} =
		\begin{bmatrix}
		(A^++B^+F_1)x^+ \\ B^-F_1x^+ + A^-x^-
		\end{bmatrix}\qquad \forall x^+ \in X^+,~\forall x^- \in X^-.
		\]
		Since $A^++B^+F_1$ is invertible, there exists $x_0^+ \in X^+$ such that 
		$\Pi b = (A^++B^+F_1)x_0^+$.
		Moreover, if we set $x_0^- := (1-F_1 x_0^+)x^-_b$, then 
		\[
		B^-F_1x_0^+ + A^-x_0^- = (F_1 x_0^+) b^- + (1-F_1 x_0^+)b^- = b^-.
		\]
		Hence $b \in  \ran (\widetilde A) = D(\widetilde A^{-1})$.
		
		Theorem~4 of \cite{Paunonen2014JDE} shows that 
		there exists $\kappa>0$ such that $\widetilde A+BF_2=  A+B(F_1+F_2)$ satisfies the conditions in (A\ref{assump:closed_loop})
		for every $F_2 \in \mathcal{L}(X,\mathbb{C})$
		with $\|F_2\| < \kappa$.
		As in the case of the structured feedback operator $F_1$, we see that 
		$F:= F_1+F_2$ also satisfies (A\ref{assump:b_F_cond}),
		by adding a small perturbation if necessary. Thus
		Theorem~\ref{thm:SD_SS} can be applied to the sampled-data system
		with the nonstructured feedback operator $F$.
	}
\end{example}

\section{Concluding remarks}
In this paper, we have analyzed robustness of strong stability
with respect to sampling. We have limited our attention to the situation
where the generator $A$ is a Riesz-spectral operator and $0 \in \sigma(A) \setminus \sigma_p(A)$.
We have presented conditions under which 
the sufficient condition for strong stability in the 
the Arendt-Batty-Lyubich-V\~u theorem is preserved 
between the original continuous-time system and
the sampled-data system
under fast sampling. Our future work is to analyze robustness of polynomial stability
with respect to sampling.
\label{sec:conclusion}

\section*{Acknowledgements}
We would like to thank the anonymous reviewers for their valuable
comments, which
helped us to
improve the presentation of this paper and 
to shorten the proofs of 
Lemmas~\ref{lem:spectal_cond}, \ref{lem:continuous_time_trans_func_bound},
and
\ref{lem:resol_T_SF}.


\begin{thebibliography}{10}
	
	\bibitem{Arendt1998}
	{\sc W.~Arendt and C.~J.~K. Batty}, {\em Tauberian theorems and stability of
		one-parameter semigroups}, Trans. Amer. Math. Soc., 309 (1988), pp.~837--852.
	
	\bibitem{Besseling2010}
	{\sc N.~Besseling and H.~Zwart}, {\em {Stability analysis in continuous and
			discrete time, using the Cayley transform}}, Integr. Equ. Oper. Theory, 68
	(2010), pp.~487--502.
	
	\bibitem{Curtain2020}
	{\sc R.~F. Curtain and H.~J. Zwart}, {\em An Introduction to
		Infinite-Dimensional Systems, A State Space Approach}, New York: Springer,
	2020.
	
	\bibitem{Eisner2010}
	{\sc T.~Eisner}, {\em Stability of Operators and Operator Semigroups}, Basel:
	Birkh\"auser, 2010.
	
	\bibitem{Engel2000}
	{\sc K.-J. Engel and R.~Nagel}, {\em One-Parameter Semigroups for Linear
		Evolution Equations}, New York: Springer, 2000.
	
	\bibitem{Gohberg1990}
	{\sc I.~Gohberg, S.~Goldberg, and M.~A. Kaashoek}, {\em Classes of Linear
		Operators, Vol. I}, Basel: Birkh\"auser, 1990.
	
	\bibitem{Gomilko1999}
	{\sc A.~M. Gomilko}, {\em {Conditions on the generator of a uniformly bounded
			$C_0$-semigroup}}, Funct. Anal. Appl., 33 (1999), pp.~294--296.
	
	\bibitem{Guo2019book}
	{\sc B.-Z. Guo and J.-M. Wang}, {\em Control of Wave and Beam PDEs: The Riesz
		Basis Approach}, Cham: Springer, 2019.
	
	\bibitem{Guo2006}
	{\sc B.-Z. Guo and H.~Zwart}, {\em {On the relation between stability of
			continuous-and discrete-time evolution equations via the Cayley transform}},
	Integr. Equ. Oper. Theory, 54 (2006), pp.~349--383.
	
	\bibitem{Hamalainen2010}
	{\sc T.~H\"am\"al\"ainen and S.~Pohjolainen}, {\em Robust regulation of
		distributed parameter systems with infinite-dimensional exosystems}, SIAM J.
	Control Optim., 48 (2010), pp.~4846--4873.
	
	\bibitem{Kang2018Automatica}
	{\sc W.~Kang and E.~Fridman}, {\em {Distributed sampled-data control of
			Kuramoto-Sivashinsky equation}}, Automatica, 95 (2018), pp.~514--524.
	
	\bibitem{Karafyllis2018}
	{\sc I.~Karafyllis and M.~Krstic}, {\em {Sampled-data boundary feedback control
			of 1-D parabolic PDEs}}, Automatica, 87 (2018), pp.~226--237.
	
	\bibitem{Ke2009SIAM}
	{\sc Z.~Ke, H.~Logemann, and R.~Rebarber}, {\em Approximate tracking and
		disturbance rejection for stable infinite-dimensional systems using
		sampled-data low-gain control}, SIAM J. Control Optim., 48 (2009),
	pp.~641--671.
	
	\bibitem{Ke2009IEEE}
	{\sc Z.~Ke, H.~Logemann, and R.~Rebarber}, {\em A sampled-data servomechanism
		for stable well-posed systems}, IEEE Trans. Automat. Control, 54 (2009),
	pp.~1123--1128.
	
	\bibitem{Ke2009SCL}
	{\sc Z.~Ke, H.~Logemann, and S.~Townley}, {\em Adaptive sampled-data integral
		control of stable infinite-dimensional linear systems}, Systems Control
	Lett., 58 (2009), pp.~233--240.
	
	\bibitem{Lin2020}
	{\sc P.~Lin, H.~Liu, and G.~Wang}, {\em Output feedback stabilization for heat
		equations with sampled-data controls}, J. Differential Equ., 268 (2020),
	pp.~5823--5854.
	
	\bibitem{Logemann2013}
	{\sc H.~Logemann}, {\em Stabilization of well-posed infinite-dimensional
		systems by dynamic sampled-data feedback}, SIAM J. Control Optim., 51 (2013),
	pp.~1203--1231.
	
	\bibitem{Logemann2003}
	{\sc H.~Logemann, R.~Rebarber, and S.~Townley}, {\em Stability of
		infinite-dimensional sampled-data systems}, Trans. Amer. Math. Soc., 355
	(2003), pp.~3301--3328.
	
	\bibitem{Logemann2005}
	{\sc H.~Logemann, R.~Rebarber, and S.~Townley}, {\em Generalized sampled-data
		stabilization of well-posed linear infinite-dimensional systems}, SIAM J.
	Control Optim., 44 (2005), pp.~1345--1369.
	
	\bibitem{Logemann1997}
	{\sc H.~Logemann and S.~Townley}, {\em Discrete-time low-gain control of
		uncertain infinite-dimensional systems}, IEEE Trans. Automat. Control, 42
	(1997), pp.~22--37.
	
	\bibitem{Lyubich1988}
	{\sc Y.~I. Lyubich and V.~Q. Ph\^ong}, {\em {Asymptotic stability of linear
			differential equations in Banach spaces}}, Studia Math., 88 (1988),
	pp.~37--42.
	
	\bibitem{Paunonen2011}
	{\sc L.~Paunonen}, {\em {Perturbation of strongly and polynomially stable
			Riesz-spectral operators}}, Systems Control Lett., 60 (2011), pp.~234--248.
	
	\bibitem{Paunonen2012SS}
	{\sc L.~Paunonen}, {\em Robustness of strongly and polynomially stable
		semigroups}, J. Funct. Anal., 263 (2012), pp.~2555--2583.
	
	\bibitem{Paunonen2013SS}
	{\sc L.~Paunonen}, {\em Robustness of polynomial stability with respect to
		unbounded perturbations}, Systems Control Lett., 62 (2013), pp.~331--337.
	
	\bibitem{Paunonen2014JDE}
	{\sc L.~Paunonen}, {\em Robustness of strong stability of semigroups}, J.
	Differential Equ., 257 (2014), pp.~4403--4436.
	
	\bibitem{Paunonen2015Springer}
	{\sc L.~Paunonen}, {\em {On robustness of strongly stable semigroups with
			spectrum on $i\mathbb {R}$}}, in Semigroups of Operators -Theory and
	Applications, Cham: Springer, 2015, pp.~105--121.
	
	\bibitem{Paunonen2015}
	{\sc L.~Paunonen}, {\em Robustness of strong stability of discrete semigroups},
	Systems Control Lett., 75 (2015), pp.~35--40.
	
	\bibitem{Rastogi2020}
	{\sc S.~Rastogi and S.~Srivastava}, {\em Strong and polynomial stability for
		delay semigroups}, J. Evol. Equ., 21 (2021), pp.~441--472.
	
	\bibitem{Rebarber1998}
	{\sc R.~Rebarber and S.~Townley}, {\em Generalized sampled data feedback
		control of distributed parameter systems}, Systems Control Lett., 34 (1998),
	pp.~229--240.
	
	\bibitem{Rebarber2002}
	{\sc R.~Rebarber and S.~Townley}, {\em {Nonrobustness of closed-loop stability
			for infinite-dimensional systems under sample and hold}}, IEEE Trans.
	Automat. Control, 47 (2002), pp.~1381--1385.
	
	\bibitem{Rebarber2006}
	{\sc R.~Rebarber and S.~Townley}, {\em {Robustness with respect to sampling for
			stabilization of Riesz spectral systems}}, IEEE Trans. Automat. Control, 51
	(2006), pp.~1519--1522.
	
	\bibitem{Rebarber2006MTNS}
	{\sc R.~Rebarber and S.~Townley}, {\em {Sampled-data control of
			infinite-dimensional systems: Recent developments and open problems}}, in
	Proc. 17th MTNS, 2006.
	
	\bibitem{Shi2000}
	{\sc D.-H. Shi and D.-X. Feng}, {\em {Characteristic conditions of the
			generation of $C_0$ semigroups in a Hilbert space}}, J. Math. Anal. Appl.,
	247 (2000), pp.~356--376.
	
	\bibitem{Tarn1988}
	{\sc T.~J. Tarn, J.~R. Zavgern, and X.~Zeng}, {\em Stabilization of
		infinite-dimensional systems with periodic gains and sampled output},
	Automatica, 24 (1988), pp.~95--99.
	
	\bibitem{Tucsnak2009}
	{\sc M.~Tucsnak and G.~Weiss}, {\em Observation and Control of Operator
		Semigroups}, Basel: Birkh\"auser, 2009.
	
	\bibitem{Wakaiki2019}
	{\sc M.~Wakaiki and H.~Sano}, {\em Sampled-data output regulation of unstable
		well-posed infinite-dimensional systems with constant reference and
		disturbance signals}, Math. Control Signals Systems, 32 (2020), pp.~43--100.
	
	\bibitem{Wakaiki2020SCL}
	{\sc M.~Wakaiki and Y.~Yamamoto}, {\em Stability analysis of perturbed
		infinite-dimensional sampled-data systems}, Systems Control Lett., 138
	(2020), pp.~1--8, Article~104652.
	
\end{thebibliography}
\end{document}